\documentclass[a4paper,10pt]{article}

\usepackage[a4paper]{geometry}
\usepackage[T1]{fontenc}

\usepackage{amsmath, amsthm, amsfonts, amssymb, esint, graphicx}

\usepackage[all]{xy}

\theoremstyle {definition}  \newtheorem {props} {Propositions} [section]
\theoremstyle {plain} \newtheorem {corollary} [props] {Corollary}
\theoremstyle {plain} \newtheorem {definition} [props] {Definition}
\theoremstyle {plain} \newtheorem {lemma} [props] {Lemma}
\theoremstyle {plain} 
\theoremstyle {plain} 
\theoremstyle {plain} \newtheorem {proposition} [props] {Proposition}
\theoremstyle {plain} \newtheorem {question} [props] {Question}
\theoremstyle {plain} \newtheorem {remark} [props] {Remark}
\theoremstyle {plain} \newtheorem {theorem} [props] {Theorem}

\newtheorem*{theorem*}{Theorem}

\newtheoremstyle{break}  % follow `plain` defaults but change HEADSPACE.
  {}   % ABOVESPACE
  {}   % BELOWSPACE
  {\itshape}  % BODYFONT
  {}       % INDENT (empty value is the same as 0pt)
  {\bfseries} % HEADFONT
  {}         % HEADPUNCT
  {\newline}  % HEADSPACE. `plain` default: {5pt plus 1pt minus 1pt}
  {#1: #3}          % CUSTOM-HEAD-SPEC
\theoremstyle{break}
\newtheorem{fact}{Fact}

\newcommand{\id}{\mathrm{Id}}

\newcommand{\rcsubset}{\subset\subset}

\newcommand{\bound}{\partial}

\newcommand{\trace}{\mathrm{tr}\ }

\newcommand{\normes}[1]{\left| #1 \right|}

\newcommand{\ddt}{\frac{d}{dt}}
\newcommand{\dmdtm}{\frac{d^-}{d t^-}}
\newcommand{\ddtt}{\frac{d^2}{d t^2}}

\newcommand{\del}{\partial}

\newcommand{\dd}{\mathrm{d}}

\newcommand{\vol}[1]{\mathrm{vol}_{#1}\ }

\newcommand{\riem}{\mathrm{Rm}\ }
\newcommand{\sect}[1]{\mathrm{K_{#1}}}
\newcommand{\ricci}{\mathrm{Rc}\ }
\newcommand{\scal}{\mathrm{Sc}\ }
\newcommand{\IC}{\mathrm{IC}\ }

\newcommand{\dvol}[1]{\mathrm{dV}_{#1}}

\newcommand{\inj}[1]{\mathrm{inj_{#1}}\ }

\newcommand{\setE}{\mathbb{E}}

\newcommand{\setN}{\mathbb{N}}
\newcommand{\setR}{\mathbb{R}}
\newcommand{\setS}{\mathbb{S}}

\newcommand{\calC}{\mathcal{C}}
\newcommand{\calD}{\mathcal{D}}

\newcommand{\calH}{\mathcal{H}}
\newcommand{\calI}{\mathcal{I}}

\newcommand{\calK}{\mathcal{K}}

\newcommand{\calM}{\mathcal{M}}
\newcommand{\calN}{\mathcal{N}}

\newcommand{\calU}{\mathcal{U}}
\newcommand{\calV}{\mathcal{V}}
\newcommand{\calW}{\mathcal{W}}
\newcommand{\calX}{\mathcal{X}}
\newcommand{\calY}{\mathcal{Y}}

\newcommand{\gothl}{\mathfrak{l}}
\newcommand{\gothm}{\mathfrak{m}}
\newcommand{\gothn}{\mathfrak{n}}

\newcommand{\tc}{\tilde{c}}
\newcommand{\td}{\tilde{d}}

\newcommand{\tf}{\tilde{f}}
\newcommand{\tg}{\tilde{g}}

\newcommand{\tB}{\tilde{B}}

\newcommand{\tS}{\tilde{S}}

\newcommand{\tU}{\tilde{U}}
\newcommand{\tV}{\tilde{V}}

\newcommand{\tlambda}{\tilde{\lambda}}

\newcommand{\bpsi}{\overline{\psi}}

\newcommand{\bPsi}{\overline{\Psi}}

\newcommand{\bt}{\bar{t}}
\newcommand{\bx}{\bar{x}}
\newcommand{\by}{\bar{y}}

\newcommand{\bK}{\bar{K}}
\newcommand{\bQ}{\bar{Q}}

\newcommand{\ux}{\underline{x}}
\newcommand{\ut}{\underline{t}}

\newcommand{\conditionC}{(\mathcal{C})}
\newcommand{\conditionCthree}{(\mathcal{C}_3)}
\newcommand{\conditionCpn}{(\mathcal{C}''_n)}
\newcommand{\conditionCpthree}{(\mathcal{C}'_3)}

\title{Short-time existence of the Ricci flow on complete, non-collapsed $3$-manifolds with Ricci curvature bounded from below}
\author{Raphael Hochard}

\AtEndDocument{\bigskip{\footnotesize%
  {Institut Math\'{e}matique de Bordeaux, Universit\'{e} de Bordeaux, UMR 5251} \par  
  {351, cours de la Lib\'{e}ration - F 33 405 Talence, France } \par 
  \textit{E-mail:} \texttt{rhochard@math.u-bordeaux.fr} \par

}}

\begin{document}

\maketitle

\begin{abstract}
 We prove that for any complete three-manifold with a lower Ricci curvature bound and a lower bound on the volume of balls of radius one, a solution to the Ricci flow exists for short 
 time. Actually our proof also yields a (non-canonical) way to flow and regularize some interior region of a non-complete initial data satisfying the aformentioned bounds.
\end{abstract}

\tableofcontents

\section{Introduction}

A natural question in the study the Ricci flow on complete, non-compact Riemannian manifold is to determine which conditions on the geometry of the initial data guaranties short-time existence of a solution. 
A classical result is Shi's theorem, which states that short-time existence holds true when the initial metric is complete and sectional curvature is bounded. As solving a PDE often requires establishing 
a priori estimates on the solutions, a related problem consists in bounding from below the existence time of the flow, as well as controlling the evolution of some geometric quantities for positive time, 
in term of the conditions imposed on the initial metric. For instance, Shi's theorem also asserts that the flow exists for a time interval $[0,\frac{c(n)^2}{\Lambda}]$ with additional bounds 
$\normes{ \riem g(t) } \leq 4 \Lambda$ where $\Lambda$ is the bound on the Riemann tensor at time $0$, and $c(n)$ a constant depending only on the dimension.

\bigskip

Conversely every time one can prove a uniform lower bound on the existence time of the Ricci flow for a family of compact manifolds satisfying given conditions, as well as uniform estimates, then it is 
natural to ask whether short-time existence holds true for a complete manifold satisfying the same conditions. Here we consider the particular set of initial conditions
\[ \begin{array}{cccc}
    \conditionCthree & \dim M = 3, & \ricci g_0(x) \geq -1, & \vol{g_0} B(x,1) \geq v_0.
   \end{array} \]
for all $x \in M$. Uniform estimates for the family of compact Riemannian manifolds satisfying $\conditionCthree$ were shown by Miles Simon

\begin{theorem*}[M. Simon, 2012 (\cite{Simon_2009}, theorem 1.11 or \cite{Simon_2012}, theorem 1.9)]
 There exist universal constants $\epsilon(v_0), K(v_0)>0$ and $k(v_0)$ such that the following holds true. Let $(M^3,g_0)$ be a compact Riemannian manifold satisfying $\conditionCthree$. Then the unique Ricci 
 flow on $M$ with initial data $g_0$ can be extended up to time $\epsilon^2$. Moreover, the estimates
 \[ \begin{array}{ccc}
  \ricci g_0(x) \geq -k, & \vol{g(t)} B_t(x,1) \geq \frac{v_0}{2}, & \normes{ \riem g(x,t) } \leq \frac{K}{t},
 \end{array} \]
 hold for all $x \in M$, $0 < t \leq \epsilon^2$.

\end{theorem*}
When $(M,g_0)$ is a complete non compact manifold, the same author proved that whenever Shi's solution is available -that is, when $g_0$ has bounded sectional curvature\footnote{in fact he could even extend the proof of short time
existence to a class of initial data with unbounded curvature, but satisfying another set of assumptions, including that all critical point of the distance-to-the-origin function lie in a compact set, see condition $(\tc)$ in 
\cite{Simon_2012}.}- it can be extended up to a uniform time $\epsilon(v_0)^2$, with similar estimates. In this paper, we essentially remove this hypothesis of bounded curvature at initial time, proving

\begin{theorem} \label{main_corollary}
 There exists $\epsilon_0(v_0)>0$ such that if $(M^3,g_0)$ is a complete Riemannian manifold satisfying $\conditionCthree$, then there exists a complete Ricci flow on $M \times [0,\epsilon_0^2]$ with $g(0)=g_0$.
\end{theorem}

The proof relies on a general construction, valid for any dimension, which provides us with a way of flowing an even a non-complete initial data in a practically useful manner. We shall call this construction, which we introduce in 
section \ref{section_results}, ``partial Ricci flow''.

Theorem $\ref{main_corollary}$ removes the remaining restrictions in Miles Simon's answer for dimension three (thm 1.7 in \cite{Simon_2012}) to a question of M. Anderson, J. Cheeger, T. Colding, G. Tian (see conjecture 0.7 in 
\cite{Cheeger_Colding_1997}):

\begin{theorem}
 Any metric space arising as the Gromov-Hausdorff limit of a sequence of complete Riemannian three-manifolds $M_i$ with $ \ricci g_i \geq -1$ and $\vol{g_i} B_{g_i}(x,1) \geq v_0$ for all $x \in M_i$, is homeomorphic to a smooth
 differential manifold.
\end{theorem}

To conclude this introduction, let us mention some alternative conditions one can impose on the initial data, and the results obtained so far (see \cite{Topping_2014} for a survey of Ricci flows with unbounded curavture). For a 
complete pointed Riemannian manifold $(M,g_0,x_0)$ we consider
\[ \begin{array}{cccc}
    \conditionCpthree & \dim M = 3, & \ricci g_0(x) \geq 0, & \vol{g_0} B(x_0,1) \geq v_0,
   \end{array} \]   
as well as 
\[ \begin{array}{cccc}
    \conditionCpn & \dim M = n, & \IC g_0(x) \geq 0, & \vol{g_0} B(x_0,1) \geq v_0,
   \end{array} \]
where $\IC$ stands for isotropic curvature (see definition in \cite{Cabezas_Rivas_2015} for example). For condition $\conditionCpn$, short-time existence including the complete non-compact case, as well as a uniform bound on 
the existence time and estimates, were established by E. Cabezas-Rivas and B. Wilking in \cite{Cabezas_Rivas_2015}. A uniform bound on the existence time and additional estimates were proved by X. Guoyi in \cite{Xu_2015} for the 
family of those compact Riemannian manifolds satisfying condition $\conditionCpthree$, yet the non-compact case remains open, whence the following

\begin{question}
 Let $(M^3,g_0,x_0)$ be a complete non-compact manifold satisfying condition $\conditionCpthree$. Does there exist a solution to the Ricci flow equation on $M \times [0,\epsilon^2]$ for some $\epsilon>0$,
 with $g(0)=g_0$ ?
\end{question}

{\bf Acknowledgements. } I would like to thank my advisor Laurent Bessi\`{e}res for submitting this problem to me, and for his support, helpful remarks and advice throughout the writing of the paper.

\section{The main construction: a partial Ricci flow} \label{section_results}

 A natural approach to producing a flow starting from a manifold $(M,g_0)$ satisfying some condition $\conditionC$, but having unbounded curvature, is to approximate the initial data by a sequence $(M_i,g_{0,i})$ of manifolds 
 either compact or complete with bounded sectional curvature still satisfying condition $\conditionC$. For each term of the sequence short time existence of the flow is a consequence either of Hamilton-DeTurk or of Shi's theorem. 
 Then if one is able to establish a lower bound on the flow's existence time for the family of compact or complete-with-bounded-curvature manifolds satisfying $\conditionC$ as well as uniform estimates, the flow of $(M,g_0)$ can 
 be obtained as the limit of the flows $(M_i,g_i(t))$ of initial data $g_{0,i}$.

 \bigskip

 This is the approach carried over in particular by E. Cabezas-Rivas and B. Wilking in \cite{Cabezas_Rivas_2015}, yet it fails for the set of conditions $\conditionCthree$ we are interested in, as we were not able to approximate 
 the initial metric by a sequence of metrics of bounded curvature, while keeping the Ricci curvature uniformly bounded from below. As a consequence, we find no uniform existence time for the flows of the terms of the sequence, and 
 we cannot extract a limit. To circumvent this, we introduce a general construction, valid in any dimension, which we call ``partial Ricci flow'' and which allows us to flow even non-complete metrics in a way meaningful for our
 purpose. Roughly speaking, if $(M^n,g)$ is a not necessarily complete $n$-dimensional Riemannian manifold, a partial flow of $M$ is a smooth solution to the Ricci flow equation defined on  an open space-time domain $\calD \subset 
 M \times [0,1]$ which contains $M \times \{ 0 \}$ as its initial time slice.
 
 \bigskip
 
 Such a construction becomes useful when one is able to control the shape of the domain $\calD$ in term of appropriate hypothesis on the initial data $g_0$. To this effect we endow the partial flow with some additional 
 properties. Recall the definition
 
 \begin{definition}
  A Riemannian metric $g$ on $M$ is said to have controlled geometry at a scale $r>0$ on an open set $U \subset M$ if $\normes{ \riem g(x) } \leq r^{-2}$ for every $x \in U$, and $\inj{g} x \geq r$ at every point 
  $x \in U$ such that $B_g(x,r)$ is relatively compact in $U$.
 \end{definition}

 For any choice of a parameter $K_0 \geq 1$, the partial flow will be constructed so that firstly, $g(t)$ has its geometry controlled a priori at scale $C^{-1} \sqrt{ \frac{t}{K_0} }$, for some constant $C>>1$, on the time slice
 $\calD_t = \{ x \in M \ | \ (x,t) \in \calD \}$, for any $0<t \leq1$. Secondly, if $0< \tau \leq 1$ is a time and $U \subset \calD_\tau$ a region such that $g(t)$ has controlled geometry at the significantly larger scale $\sqrt{ 
 \frac{t}{K_0} }$ on $U$ for any $0<t \leq \tau$, then an additional space-time domain $(U)_{ \Delta \rho} \times [\tau, \tau+\Delta \tau]$ is included in $\calD$. (Where, for $U \subset M$ and $r>0$,
 \[ (U)_r = \{ x \in U \ | \ B_{g_0}(x,r) \text{ is relatively compact in } U \}. \]
 From now on, we shall use this notation whenever there is no ambiguity as to the metric relative to which such an $r$-interior region is taken). Moreover, the size of margin $\Delta \rho$ and the additional time interval 
 $\Delta \tau$ are controlled in term of the parameter $K_0$.
 
 \begin{center}
  \includegraphics{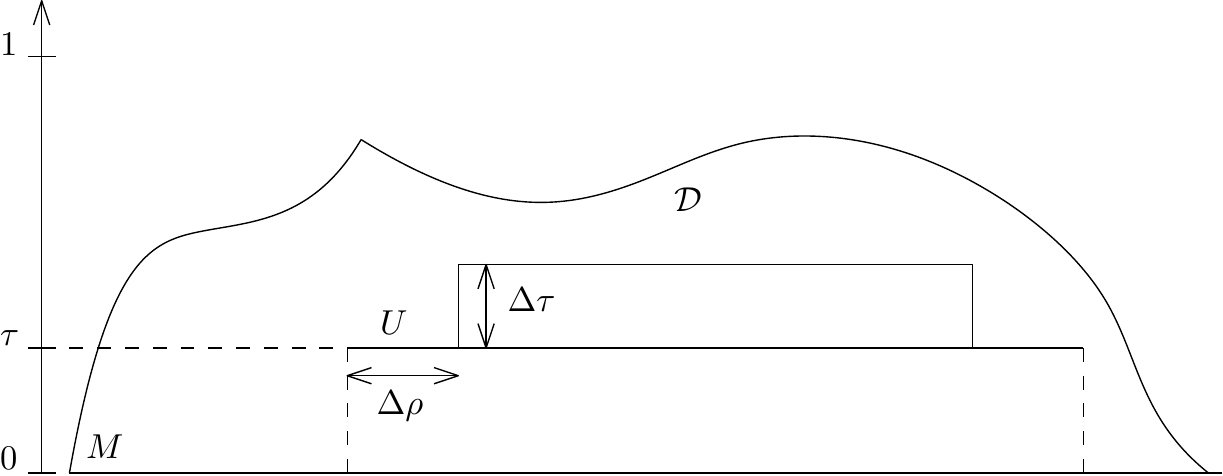}
 \end{center}

 These properties are summed up into the following existence theorem for the partial flow. 
 
 \begin{theorem}[Partial Ricci flow] \label{theorem_partial_flow}
 There exists a constant $C(n)$ with the following property. Let $(M^n,g_0)$ be a Riemannian manifold. For any choice of a parameter $K_0 \geq 1$, there exists a smooth solution $g(x,t)$ to the Ricci flow equation defined on 
 an open domain $\calD \subset M \times [0,1]$ with
  \begin{itemize}
   \item[(i)] $\calD_0 = M$, $\calD_s \supset \calD_t$ for any $0 \leq s \leq t \leq 1$,
   \item[(ii)] $g(t)$ has controlled geometry at scale $C^{-1} \sqrt{\frac{t}{K_0}}$ on $\calD_t$ for every $0<t \leq 1$, 
   \item[(iii)] For any $0<\tau \leq 1$, if $U \subset \calD_\tau$ is an open domain such that for every $0 < t \leq \tau$, $g(t)$ has controlled geometry at scale $\sqrt{\frac{t}{K_0}}$ on $U$, then the additional space-time 
   domain $(U)_{\Delta \rho} \times [\tau,\min(\tau+\Delta \tau,1)]$ is contained in $\calD$, with $\Delta \rho = C \sqrt{K_0 \tau}$, $\Delta \tau = \frac{\tau}{C^2 K_0}$.
  \end{itemize}
\end{theorem}

 The proof of this result is carried out in section $\ref{section_partial_flow}$.
 
 \bigskip
 
 The main result of this paper (of which theorem $\ref{main_corollary}$ is a special case) essentially says the following: if a (not necessarily complete) Riemannian manifold $(M^3,g_0)$ satisfies condition $\conditionCthree$, 
 then any partial flow of $(M,g_0)$ (for an appropriate choice of the parameter $K_0$) contains in its domain of definition $\calD$ a definite space-time region of the form
 \[ \calM_{A,\epsilon^2} := \{ (x,t) \in M \times [0, \epsilon^2] \ | \ B_{g_0}(x,A \sqrt{t}) \text{ is relatively compact in } M \} \]
 the parameters $A, \epsilon$ depending only on the constant $v_0$ which appears in condition $\conditionCthree$. In particular, a smooth non-complete Ricci flow exists on the domain $(M)_{A \epsilon} \times [0,\epsilon^2]$, and 
 thus, in the case when $M$ is complete, on $M \times [0,\epsilon^2]$. The full statement is 

\begin{theorem} \label{main_theorem}
 There exist functions $K(v_0)$, $v(v_0)>0$ $\epsilon_0(v_0)>0$ and $A(v_0)>0$,  such that the following holds. Let $(M^3,g_0)$ be a three-manifold equipped with a Riemannian metric which needs not be complete, satisfying the 
 following hypothesis
 \begin{align*}
  & \ricci g_0 \geq -1 \text{ on } M, \\
  & \vol{g_0} B(x,r) \geq v_0 r^3,
 \end{align*}
 for all $x \in M$, $0<r<1$ such that $B(x,r)$ is relatively compact $M$. Then there exists a solution $g(x,t)$ to the Ricci flow equation defined on $\calM_{A,\epsilon^2}$, with non-complete time slices in general, such that 
 \begin{align*}  
  & \normes{\riem g(x,t)} \leq \frac{K}{t}, \\
  & \ricci g(x,t) \geq - \frac{1}{t}, \\
  & \vol{g(t)} B_t(x,r) \geq v r^3,
 \end{align*}
 for every $(x,t) \in \calM_{A,\epsilon^2}$, and $0<r\leq \sqrt{t}$ such that $B_t(x,r)$ is relatively compact in $\calM_{A,\epsilon^2}$. 
\end{theorem}

\bigskip

Finally, the following (three-dimensional) pseudo-locality type statement, which extends theorem 1.1 in \cite{Simon_2014} to the case when Ricci curvature (instead of sectional curvature) is bounded from below, comes as a 
byproduct of the proof of theorem $\ref{main_theorem}$.

 \begin{theorem} \label{three_dim_pl}
  There exist $\epsilon_{\ref{three_dim_pl}}(v_0), A(v_0), K(v_0), v(v_0)>0$ such that the following holds. Let $(M,g(t))$ be a complete Ricci flow on $[0,T]$ such that 
  \begin{align*}  
   & \sup_{M \times [t,T]} \normes{ \riem g(x,s) } < + \infty 
   \intertext{for all $0 < t \leq T$, and let $U \subset M$ be an open region where}
   & \ricci g(x,0) \geq -1, \\
   & \vol{} B_0(x,r) \geq v_0 r^3,
  \intertext{for every $0 \leq r \leq 1$ and $x \in U$. Then for any $0 < \epsilon \leq \epsilon_{\ref{three_dim_pl}}$,}
   & \normes{ \riem g(x,t) } \leq \frac{K}{t}, \\
   & \ricci g(x,t) \geq - \frac{1}{t}, \\
   & \vol{} B_t(x,r) \geq v r^3 \text{ for } 0 \leq r \leq \sqrt{t},
  \end{align*}
  for every $x \in (U)_{A \epsilon}$ where $(U)_{A \epsilon} = \{ x \in U \ | \ B_0(x,A \epsilon) \subset U \}$ and $0 < t \leq \max(\epsilon^2,T)$.
 \end{theorem}

\subsection{Outline of the proof of theorem $\ref{main_theorem}$}

Consider a Riemannian manifold $(M^3,g_0)$ satisfying conditions $\conditionCthree$. The shape of the domain $\calD$ on which a partial flow of $(M,g_0)$ is defined is controlled trough the use of the maximality property of the 
construction. Precisely, we need to show that the condition $\conditionCthree$ imposed on the initial metrics, combined with the a priori control at scale $C^{-1} \sqrt{\frac{t}{K_0}}$ on the geometry of $g(t)$, imply that the 
geometry is actually controlled at the enhanced scale $\sqrt{\frac{t}{K_0}}$, if $K_0$ is chosen adequately. This is done through the following 

\begin{proposition} \label{main_proposition}

 There exists $\epsilon_0(v_0, K_1), a(v_0,K_1), v(v_0)>0$, $K(v_0)$ with the following property. Let $g(t)$ be a Ricci flow (not necessarily complete) defined on $U \times [0, \epsilon^2]$ such that at initial time
 \begin{align*}
  & \ricci g(x,0) \geq -1, \\
  & \vol{g(0)} B_0(x,r) \geq v_0 r^3 \text{ whenever } r \leq 1 \text{ and } B_0(x,r) \text{is relatively compact in } U. 
 \intertext{Suppose moreover that at each time $0<t \leq \epsilon_0^2$,}
  & \normes{ \riem g(x,t) } \leq \frac{K_1}{t}, \text{ for } x \in U, \\
  & \inj{g(t)} x \geq \sqrt{\frac{t}{K_1}} \text{ as long as } B_t(x,\sqrt{\frac{t}{K_1}}) \text{ is relatively compact in } U.
 \intertext{Then}
  & \normes{ \riem g(x,t) } \leq \frac{K}{t}, \\
  & \ricci g(x,t) \geq -\frac{1}{t}, \\
  & \vol{g(t)} B_t(x,r) \geq v r^3 \text{ for all } 0< r \leq \sqrt{t}.
 \end{align*}
 for $x \in (U)_{a \sqrt{t}}$, $t \in ]0, \epsilon_0^2]$,  where $(U)_{a \sqrt{t}} = \{ x \in U \ | \ B_0(x,a \sqrt{t}) \text{ is relatively compact in } U \}$.
 
\end{proposition}

\bigskip

Let thus $K_0=K_0(v_0)$ given by proposition \ref{main_proposition}, and consider a partial flow $g(t)$ of parameter $K_0$ with initial data $(M,g_0)$. Applying proposition $\ref{main_proposition}$ to $U = \calD_t$ with $K_1 = 
C^2 K_0$, we find that the geometry of $g(t)$ is actually controlled at scale $\sqrt{\frac{t}{K_0}}$ on $(U)_{a \sqrt{t}}$ for any $0<t \epsilon^2$. By the maximality property of the construction mentioned above, this implies that 
the space-time region $(U)_{\Delta \rho} \times [t, t+ \Delta t]$ is contained in $\calD$, for $\Delta \rho = a' \sqrt{t}$, $\Delta t=\frac{c^2 t}{K_1}$ (where $a'=a+ C\sqrt{K_0}$). Fixing $\calK \subset M$ compact, $t < \epsilon$ 
and starting for some $t_0$ sufficiently small so that $K \times [0,t_0] \subset \calD$, we define sequences $t_k$, $\rho_k$ by $t_{k+1}=t_k + \Delta t_k$ with $\Delta t_k = \frac{1}{C^2 K_0} t_k$, $\rho_0=0$, $\rho_{k+1} = \rho_k 
+ \Delta \rho_k$ with $\Delta \rho_k = a' \sqrt{t_k}$. Then we have that for each $k$ such that $t_k \leq \epsilon^2$, $ (\calK)_{\rho_k} \times [0, t_k] \subset \calD$. Now $\frac{\Delta \rho_k}{\Delta t_k} = 
\frac{a' C^2 K_0}{\sqrt {t_k} }$ which integrates into $\rho_k \leq 2 a' C^2 K_0 \sqrt{t_k}$. Starting from $t_0 = (1+\frac{1}{C^2 K_1})^{-k_0}$ for $k_0$ large enough and considering the sequences up to index $k_0$, we find 
$(\calK)_{A \sqrt{t}} \times [0,t] \subset \calD$ for $A=2 a' C^2 K_0$. Since this is valid for any compact $\calK$, $t \leq \epsilon^2$, we get $\calM_{A,\epsilon^2} \subset \calD$.

\begin{center}
 \includegraphics{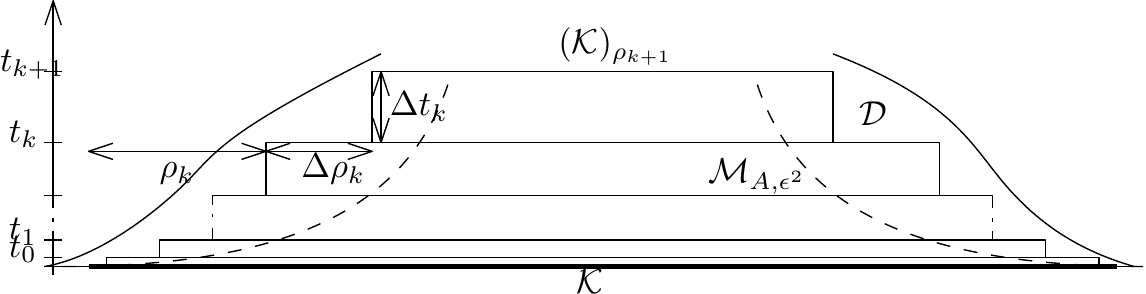}
\end{center}

Let us now attempt to sketch the ideas underlying the proof of proposition $\ref{main_proposition}$. Roughly speaking, this proposition means that when a region has Ricci curvature bounded from below and is non-collapsed at 
initial time, it remains so, and does not develop high curvature regions (i.e. singularities) for a definite time interval.

\bigskip

Singularities in three dimensional flow are known to occur in the form of ``neckpinches'' or ``caps'', in particular, balls centered around the point where the singularity forms become arbitrarily collapsed when curvature 
blows up. This is a consequence of the study by Perelman of ancient, positively curved Ricci flows, (namely of the fact that such a flow is either the trivial solution given by the Euclidian metric on $\setR^3$ or has zero 
Asymptotic Volume Ratio\footnote{\label{page_AVR}Recall that the Asymptotic Volume Ratio, or AVR, of a $n$-manifold $(M,x_0)$ of non-negative Ricci curvature is defined as the limit when $R \rightarrow + \infty$ of the ratio of the volume of the ball
$B(x_0,R)$ in $M$ over that of the ball of same radius in Euclidian space (this ratio is monotonous decreasing w.r.t $R$). It is independent of the choice of $x_0 \in M$.} - see \cite{Perelman_2002}, 11 and \cite{Kleiner_Lott}, 
I.11) when applied to a blown-up limit flow of a singularity. 

The contrapositive of this principle (formulated precisely as proposition $\ref{main_II}$) implies that there exists a function $K(v)>0$ such that if on a region $\calU$, balls of radius less than $\sqrt{t}$ are $v$-non-collapsed 
at time $t$, then 
\begin{equation} \label{label_intro_3}
 \normes{ \riem g(t) } \leq \frac{K(v)}{t}
\end{equation}
on a slightly smaller region. In other words, non-collapsing at positive times implies a bound on the Riemann tensor.

\bigskip

Consider now a flow $g(t)$ defined on a time interval $[0,\bt]$ on a region $\calU$ of a three manifold for which we have at initial time
\begin{equation} \label{noncollapsing_hypothesis}
 \begin{array}{cc} \ricci g(0) \geq -1, & \vol{g(0)} B_0(x,1) \geq v_0 \end{array} 
\end{equation}
and suppose the estimate $\normes{\riem g(x,t)} \leq \frac{K}{t}$ holds up to time $\bt$. Classically, the contraction of distances under the flow is controlled by
\[ d_t \geq d_0 - \frac{40}{3} \sqrt{K t} \]
on a slightly smaller region $\calV$. Also classical is the ``local pinching'' of the scalar curvature (proposition $\ref{proposition_minoration_scalaire}$), which guaranties that it remains bounded by 
(say) $-2$ for a definite amount of time on $\calV$. Since scalar curvature controls the evolution of the volume element under the flow, this implies in particular
\[ \dvol{g(t)} \leq 2 \dvol{g(0)} \]
on $\calV$ for $0 \leq t  \leq \bt$, if $\bt$ is small enough.

\bigskip

Consider the following elementary fact concerning maps that do not contract distances and have bounded volume dilatation

\begin{fact} \label{label_intro_2}
 There exists a function $v(n,I)>0$ such that the following holds. Let $(N_0,h_0,x_0)$ be an $n$-dimensional Riemannian manifold such that $B(x_0,1)$ is relatively compact in $N_0$, and such that an isoperimetric 
 inequality
 \[ \vol{n} U \leq I \left( \vol{n-1} \del U \right)^{\frac{n}{n-1}} \]
 holds for any $U \subset B(x_0,1)$. If there exists a smooth embedding $f: B(x_0,1) \rightarrow N_1$ of $B(x_0,1) \subset N_0$ into an $n$-dimensional Riemannian manifold $(N_1,h_1)$ such that
 \[ \begin{array}{cc} d_{h_1}(f(x),f(x')) \geq d_{h_0}(x,x'), & \vol{h_1} f(U) \leq 2 \vol{h_0} U \end{array}, \]
 then the manifold $N_1$ is non-collapsed near the image of $f$, that is, $\vol{h_1} B(f(x_0),1) \geq v$.
\end{fact}

Now fix some time $0 < t \leq \bt$ and suppose $x$ is such that both $B_0=B_0(x,\Lambda \sqrt{t})$ and $B_1=B_t(x,\Lambda \sqrt{t})$ are contained in $\calV$, for some large $\Lambda$. Then the identity 
map between $B_0$ with metric $h_0 = \frac{1}{\Lambda^2 t} g(0)$ and $B_1$ with metric $h_1 = \frac{1}{\Lambda^2 t} g_1$ satisfies $d_{h_1} \geq d_{h_0} - \frac{40}{3} \frac{\sqrt{K}}{\Lambda}$ and 
$\dvol{h_1} \leq 2 \dvol{h_0}$. Moreover, from $(\ref{noncollapsing_hypothesis})$ an isoperimetric inequality with constant $I(v_0)$ holds on $B_0$ if $\Lambda \sqrt{t} \leq 1$. Thus asymptotically when $\Lambda$ is large, we find 
ourselves almost in the situation of the fact stated above, which should allow us to deduce a lower bound $v(v_0)$ on the volume of $B_1$. With the technical lemma $\ref{technical_lemma}$, we make rigorous sense of this argument, 
and show that under appropriate additional assumptions, there exists a constant $\Lambda(n, K, v_0)$ such that if $t \leq \frac{1}{\Lambda^2}$, we have 
\[ \vol{g(t)} B(x, \Lambda \sqrt{t}) \geq v \ (\Lambda \sqrt{t})^3, \]
i.e. balls at scale $\Lambda \sqrt{t}$ remain non-collapsed for a definite amount of time.

\bigskip

Ricci flow on a complete three-manifold preserves non-negative Ricci curvature. Here we make use of a local pinching result for Ricci curvature in dimension three (corollary $\ref{ricci_local}$), which states that still under the 
hypothesis $\normes{ \riem g(x,t) } \leq \frac{K}{t}$ on a region $\calU \times [0,t]$, and supposing $\ricci g(0) \geq -1$ on $\calU$, for any $\alpha>0$ one can show that
\begin{equation} \label{intro_label_1}
 \ricci g(t) \geq -\frac{\alpha}{t}
\end{equation}
on a smaller region $\calV$, for $0 < t \leq \epsilon^2$, where $\epsilon=\epsilon(K, \alpha)$.

\bigskip

Consider a flow $g(t)$ on a region $\calU$ whose initial data satisfies hypothesis $(\ref{noncollapsing_hypothesis})$, fix $K$ and let $(\bx,\bt)$ be a point where the $\frac{K}{t}$ curvature bound is first violated - for 
simplicity say we have $\normes{ \riem g(x,t) } \leq \frac{K}{t}$ on $\calU \times [0,\bt]$, while $\normes{ \riem g(\bx,\bt) } = \sup_{B_{\bt}(\bx, \sqrt{\bt})} \normes{ \riem g(\bt) } = \frac{K}{\bt}$ where the point $\bx$ 
belongs to a slightly smaller region $\calV \subset \calU$. By the discussion above, if $t \leq \frac{1}{\Lambda^2}$, we get $v(v_0)$-non-collapsing at a scale $\Lambda \sqrt{t}$, i.e.
\[ \vol{\bt} B_{\bt}(x, \Lambda \sqrt{\bt}) \geq v(v_0) (\Lambda \sqrt{\bt})^3, \]
for $x$ in a region $\calV \subset \calW \subset \calU$. Then applying $(\ref{intro_label_1})$ with $\alpha = \Lambda^{-2}$, we deduce (if $t \leq \min(\Lambda^{-2},\epsilon^2))$, simply by Bishop-Gromov volume comparison, 
non-collapsing at all scales below, in particular, say, 
\[ \vol{\bt} B_{\bt}(\bx, \sqrt{\bt}) \geq \frac{v(v_0)}{100} \sqrt{\bt}^3. \]
This in turn would imply $\riem{ g(\bx,\bt) } \leq \frac{K_0}{\bt}$, where $K_0=K(\frac{v}{100})$, and $K(.)$ is the function introduced at $(\ref{label_intro_3})$. Thus if we make the choice $K>K_0$ in the first place,  we get a 
contradiction. Finally the $\frac{K}{t}$ bound on the norm of the Riemann tensor cannot be violated for the small but definite timespan $t \leq \min(\Lambda^{-2}, \epsilon^2)$ on $\calV$.

\begin{center}
 \includegraphics{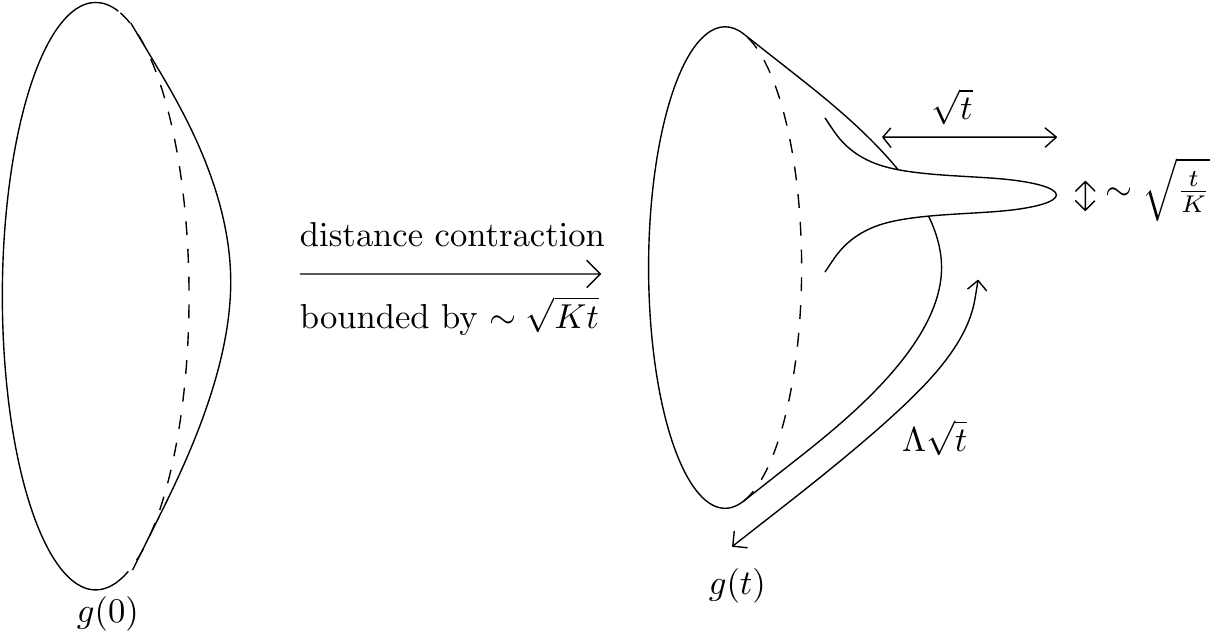}
\end{center}

The organization of the paper is as follows: sections 3 to 5 are dedicated to the proof of proposition $\ref{main_proposition}$. Section $\ref{section_partial_flow}$ is concerned with the actual construction of the partial flow,
while main theorem $\ref{main_theorem}$ is proven in section 7. Finally Appendices A and B collect known results on local minimum principles, and regularity scale control under the Ricci flow respectively.

\section{Evolution of geometric quantities}

In this section we recall well-known estimates on the evolution of distances and volumes under the Ricci flow. Throughout the paper, Ricci flows and Riemannian manifolds are not supposed complete unless 
explicitly stated.

\subsection{Distances and volumes}

The following results on the evolution of the distance function associated with a metric evolving by the Ricci flow are due to Hamilton. We simply state them under a form well-suited to the context of non-complete Ricci flows. 
Recall that for a function $f: \setR \rightarrow \setR$, one defines $\left. \dmdtm \right|_{t=t_0} f = \liminf_{t \rightarrow t_0} \frac{f(t)-f(t_0)}{t-t_0}$.

\begin{lemma} 
 Let $g(t)$ be a Ricci flow on $M^n \times [0,T]$. Let $x_0, x_1 \in M$ be such that, at time $t_0$, $d_{t_0}(x_0,x_1)$ is realized by some minimizing geodesic $\gamma \subset M$. If the curvature bound
 \[ \ricci g(x,t_0) \leq (n-1) C \]
 holds at time $t_0$ for every $x \in ( B_{t_0}(x_0,r_0) \cap \gamma ) \cup ( B_{t_0}(x_1,r_0) \cap \gamma )$, then
 \[ \left. \dmdtm \right|_{t=t_0} d_t(x_0,x_1) \geq -2 (n-1) \left( \frac{2}{3} C r_0 + \frac{1}{r_0} \right). \]
\end{lemma}
(for a proof of this fact, see for example \cite{Kleiner_Lott}, I.8.3)

Note that if $U \subset M$ is a relatively compact open set, then for $x \in U$, the distance $d_t(x, M \setminus U)$ is always realized by a minimizing geodesic of $M$, which lies entirely in $U$, except for one of its endpoints. 
Thus the above implies the following distance contraction control.

\begin{corollary} \label{lemme_distorsion_distances}
 Let $(M,g(t))$ be a Ricci flow on $[0,T]$ and $U \subset M$ a relatively compact open region such that
 \[ \normes{ \riem g(x,t) } \leq \frac{K}{t} \]
 for all $x \in U$, $0 < t \leq T$. Then for every $x \in U$, $ 0 \leq t \leq t' \leq T$ one has
 \[ d_{t'}(x,M \setminus U) \geq d_{t}(x,M \setminus U) - \frac{20}{3} (n-1) \sqrt{K} \left( \sqrt{t'}-\sqrt{t} \right). \]
\end{corollary}

We finally retrieve the classical distance-between points comparison principle

\begin{corollary} \label{corollaire_distorsion_distances}
 Let $g(t)$ be a Ricci flow on $M^n \times [0,T]$, and let $0 \leq t \leq T$, $x_0 \in M$ and $a>0$ be such that $B_t(x_0, 3a)$ is relatively compact in $M$, and such that the curvature estimate
 \[ \normes{\riem g(x,t')} \leq \frac{K}{t'} \]
 holds for any $x \in B_t(x_0,3a)$ and $t < t' \leq T$ (where $K>1$). Then for any $x, x' \in B_t(x_0,a)$ and $t \leq t' \leq T$, one has
 \[ d_{t'}(x,x') \geq d_t(x,x') - \frac{20}{3} (n-1) \sqrt{K} \left( \sqrt{t'}-\sqrt{t} \right). \]
\end{corollary}

\begin{proof} Simply fix $0 \leq t \leq T$ and apply corollary $\ref{lemme_distorsion_distances}$ to $U=B_t(x,d_t(x,x') \subset B_t(x_0,3a)$. \end{proof}

Recall that the volume elements of a metric $g(t)$ evolving under the flow obeys the equation
\[ \ddt \dd \vol{g(t)} = - \scal g(x,t) \dd \vol{g(t)} \]
In particular, if $\scal g(x,t)$ is uniformly bounded on a region $U \in M$, i.e.
\[ \scal g(x,t) \geq - \epsilon^2 \]
for every $x \in U$, $0 \leq t' \leq t$, then
\[ \vol{g(t)} U \leq e^{ \epsilon^2 t} \vol{g(0)} U. \]

The classical pinching behavior $\scal g(x,t) \geq \frac{1}{\inf_x \scal g(x,0) - \frac{2}{n} t}$ of the scalar curvature for complete flows admits a local version, due to B. L. Chen \cite{Chen_2009}. In particular, a lower bound 
on scalar curvature at initial time in a region implies a lower bound at further time on an interior region.

\begin{proposition} \label{proposition_minoration_scalaire}
 Let $g(t)$ be a Ricci flow on $(M,x_0) \times [0,T]$ such that the following holds:
 \begin{align*}
  & B_t(x_0, a+A) \text{ is relatively compact in } M, \\
  & \normes{ \riem g(t) } \leq \frac{K}{t} \text{ for } x \in B_t(x_0,\sqrt{\frac{t}{K}}), \\
  & \scal g(0) \geq - \epsilon^2 \text{ on } B_0(x,a+A).
 \end{align*}
 Then 
 \[ \scal g(x,t) \geq -\max(\epsilon^2, \frac{100}{A^2}) \]
 for every $(x,t)$ with $0<t \leq T$, $d_t(x_0,x) \leq a - \frac{20}{3} (n-1) \sqrt{K t}$.
\end{proposition}

\subsection{Pinching of the Ricci tensor}

Recall that the non-negativity of the Ricci tensor is preserved by a complete Ricci flow in dimension three (see \cite{Chen_2013}). While one cannot expect the exact analogue of the local property 
$\ref{proposition_minoration_scalaire}$ to hold for the lowest eigenvalue of the Ricci tensor even in dimension three, the following Hamilton-Ivey inequality was proved by Z.H. Zhang \cite{Zhang_2015}, and is local in nature. (see 
Appendix \ref{appendix_local} for details)

\begin{proposition} \label{ricci_pinching}
 Let $g(t)$ be a Ricci flow on $(M^3,x_0) \times [0,T]$ such that the following holds
 \begin{align*}
  & B_t(x_0, a+ A) \text{ is relatively compact in } M, \\
  & \normes{ \riem g(t) } \leq \frac{K}{t} \text{ for } x \in B_t(x_0,\sqrt{\frac{t}{K}}), \\
  & \ricci g(0) \geq - g(0) \text{ on } B_0(x,a+ A).
 \end{align*}
 Then 
 \[ \frac{\scal}{-\lambda_1} - \ln (-\lambda_1)(1+t) + 3 \geq -\frac{200 (1+t)}{A^2} \]
 for every $(x,t)$ with $0<t \leq T$, $d_t(x_0,x) \leq a - 20 \sqrt{K t}$, and $\lambda_1(x,t)<0$ where $\lambda_1(x,t)$ denotes the lowest eigenvalue of the Ricci tensor at $(x,t)$.
\end{proposition}

In the case when the initial data has Ricci curvature bounded from below, and the flow satisfies curvature estimates of the form
\[ \normes{ \riem g(x,t) } \leq \frac{K}{t} \]
one deduces the following lower bound on the Ricci curvature

\begin{corollary} \label{ricci_local}
  There exists $\epsilon_1(\alpha,K)$, such that if $g(t)$ is a Ricci flow on $(M^3,x_0) \times [0,1]$ with the following properties
 \begin{align*}
  & B_t(x_0, a+ \frac{1}{\epsilon_1}) \text{ is relatively compact in } M, \\
  & \normes{ \riem g(t) } \leq \frac{K}{t} \text{ for } x \in B_t(x_0,a), \\
  & \ricci g(0) \geq - \epsilon_1^2 \text{ on } B_0(x,a+ \frac{1}{\epsilon_1}).
 \end{align*}
 then for every $0 < t \leq 1$,
 \[ \ricci g(x,t) \geq - \frac{\alpha}{t} \]
 holds for $x \in B_t(x_0,a-20 \sqrt{K t})$.
\end{corollary}

\begin{proof}
 Let us consider, for some $\epsilon>0$ to be fixed later, the flow
  \[ h(s) = \epsilon^2 g(\frac{s}{\epsilon^2}) \]
 defined for $0 \leq s \leq \epsilon^2$. Suppose that for some $A>1$ to be determined, the following holds: for every $0  \leq t \leq 1$, $B_{g(t)}(x_0,a+\frac{A}{\epsilon})$ is relatively compact in 
 $B_{g(0)}(x_0,a+\frac{A}{\epsilon})$, $\normes{ \riem g(x,t) } < \frac{K}{t}$ on $B_{g(0)}(x_0,a)$ and $\ricci g(0) \geq - \epsilon^2$. The flow $h$ satisfies the hypothesis of proposition
 $(\ref{ricci_pinching})$. Thus if $(x,s)$ is such that $x \in B_{h(s)}(x_0,\epsilon a - 20 \sqrt{K s})$, $0 \leq s \leq \epsilon^2$, and such $\tlambda_1 <0$, where $\tlambda_1$ denotes the lowest
 eigenvalue of the Ricci tensor of $h$ at $(x,s)$, then
 \[ \frac{\scal h}{-\tlambda_1} - \ln (-\tlambda_1)(1+s)  \geq -3-\frac{200 (1+s)}{A^2} \]
 Choosing now $A=20$, and since $\normes{\riem h(x,s)} \leq \frac{K}{s}$, the above inequality can be rewritten
 \[ \frac{K}{-s \tlambda_1} \geq -4+\ln (-\tlambda_1(1+s)), \]
 then
 \[ \left( 1 + \frac{1}{s} \right) K \geq (-\tlambda_1)(1+s) \left( -4 + \ln ( (-\tlambda_1)(1+s) ) \right). \] 
 Let us consider the one-variable function
 \[ \begin{array}{cc} \psi: & x \rightarrow x (\ln(x) - 4) \end{array}, \]
 is an increasing bijection between $[e^4,+\infty[$ and $\setR^+$. For every $\alpha>0$, there exists $x_\alpha > e^4$ such that $\frac{\psi^{-1}(x)}{x} \leq \alpha$ as soon as $x \geq x_\alpha$ 
 (set for instance $x_\alpha = \frac{1}{\alpha} e^{\frac{1}{\alpha}+4}$). In particular,
 \[ -\tlambda_1 (1+s) \leq \psi^{-1}( (1+\frac{1}{s}) K) \leq \alpha (1 + \frac{1}{s}) \]
 as long as $(1+\frac{1}{s}) K \geq x_{\frac{\alpha}{K}}$, that is, $s \leq \frac{K}{ x_{\frac{\alpha}{K}} }$. Thus if we fix $\epsilon=\sqrt{K x_{\frac{\alpha}{K}}}$, $\epsilon_1=\frac{\epsilon}{A}$, we 
 find
 \[ \ricci h (x,s) \geq -\frac{\alpha}{s} \]
 for every $(x,s)$ with $x \in B_{h(s)}(x_0, \epsilon a - 20 \sqrt{K s})$, $0 \leq s \leq \epsilon^2$, or, in terms of the flow $g$, 
 \[ \ricci g(x,t) \geq - \frac{\alpha}{t} \]
 for every $x \in B_{g(t)}(x_0,a-20 \sqrt{K t})$, with $0 \leq t \leq 1$. (one can check that $\epsilon_1 = \frac{c}{\sqrt{\alpha}} e^{-\frac{K}{2 \alpha}}$ works) 
\end{proof}

\section{A metric lemma}

In this section we prove a technical lemma, which is essentially a quantitative version of fact $\ref{label_intro_2}$, when the non-contraction $d(f(x),f(x')) \geq d(x,x')$ assumption is weakened to 
$d(f(x),f(x')) \geq d(x,x') - \delta$ for some small $\delta$.

 \begin{lemma} \label{technical_lemma}
  There exist $v(n,\calI_0,A)>0$ and $\delta(n,\calI_0,A,\eta)>0$ such that the following holds.
  Let $(M,g,x_0)$ be a Riemannian manifold such that
 \[ \begin{array}{ll}
  (a) & B(x_0, 1) \text{ is relatively compact in $M$,} \\
  (b) & \ricci g \geq -\frac{1}{\eta^2} \text{ on $B(x_0,1)$, } \\
  (c) & \vol{n} U \leq \calI_0 \left( \vol{n-1} \bound U \right)^{\frac{n}{n-1}} \text{ for any } U \subset B(x_0,1). 
 \end{array} \]
  Let $(N,h,y_0)$ be a Riemannian manifold such that
 \[ \begin{array}{ll}
  (d) & B(y_0, 1) \text{ is relatively compact in $N$,} \\
  (e) & \ricci h \geq -\frac{1}{\eta^2} \text{ on $B(y_0,1)$, } \\
  (f) & \vol{h} B(y,r) \geq \frac{\omega_n}{2} r^n \text{ for every $y \in N$, $0 < r < \eta \delta$.}
 \end{array} \]
 Let finally $\psi: (B(x_0,1),x_0) \rightarrow (N,y_0)$ be an embedding with 
 \[ \begin{array}{ll}
  (g) & d_N(\psi(x),\psi(x')) \geq d_M(x,x') - \delta \text{ for every $x,x' \in B(x_0,1)$,} \\
  (h) & \vol{h} \psi(U) \leq A \vol{g} U \text{ for any $U \subset B(x_0,1)$, } \\
 \end{array} \] then
 \[ \vol{h} B(y_0,r) \geq v r^n \text{ for any } 0 < r \leq 1.\]
 moreover $d_N(\psi(x),\psi(x')) \leq \frac{c(n) A}{\eta^{n-1}} \left( d_M(x,x') + \delta \right)$.
 
 \end{lemma}
 
 \begin{remark}
  As announced, the conclusion of the lemma is the analogue of that of fact $\ref{label_intro_2}$, the assumption that $f$ be non-contracting being weakened to the $\delta$-non-contraction property $(g)$. In compensation, we impose
  (in addition to Ricci curvature lower bounds) a $\eta \delta$-regularity scale assumption $(f)$ on the goal metric $h$. This regularity scale can be much smaller than the non-contraction scale $\delta$. Indeed, recall that for
  our purpose, the non contraction scale $\delta$ is of magnitude $\sqrt{Kt}$, while the Euclidian regularity scale of the flow is of magnitude $\sqrt{\frac{t}{K}}$, so $\eta$ is of magnitude $\frac{1}{K}$. 
  Thus, while the maximum $\delta$ for which the conclusion holds can depend on the ratio $\eta$ between both scales (this will correspond to an upper bound on $t$), it is essential that the volume lower bound $v$ does not.
 \end{remark}

  \begin{proof}
 
  We start with several observations. Let $(M,x_0)$, $(N,x_0)$ be manifolds satisfying hypothesis $(a)$ through $(h)$ for given $1> \eta >> \delta > 0$. Then
  \begin{itemize}
  
   \bigskip
   
   \item[\bf(i)] The isoperimetric control on $B(x_0,1)$ implies a lower bound on the volume of balls. Namely for every $x \in B(x_0,1-r)$, 
   \begin{equation}
    \vol{} B(x,r) \geq v_0(n,\calI_0) r^n  \label{minoration_v0}
   \end{equation}
   where $v_0$ only depends on $n$ and on $\calI_0$. Moreover the upper bounds $\vol{M} B(x,r) \leq C(n) r^n$, $\vol{N} B(y,r) \leq C(n) r^n$ for $x \in B(x_0,1)$, $y \in B(y_0,1)$ 
   and $r \leq \eta$ are obvious consequences of $(b)$ and $(e)$.

   \bigskip

   \item[\bf(ii)] For $x,x' \in B(x_0,\frac{1-\delta}{2})$, we have the following distances estimate, depending on $\eta$.
   \begin{equation}
    d_N(\psi(x),\psi(x')) \leq \lambda(n,A,\eta) \left( d_M(x,x') + \delta \right). \label{psi_haussdorf_lipschitz}
   \end{equation}
   with $\lambda(n,A,\eta) = \frac{c(n) A}{\eta^{n-1}}$.
   To see this, consider $x,x' \in B(x_0,1 - \delta )$ with $d_M(x,x') \leq \delta$, and $\gamma \subset N$ the image by $\psi$ of a minimizing geodesic $\overline{x x'}$ (in particular, $\overline{xx'} \subset B(x_0,1)$). Let 
   $\{y_i\}_{i \in I}$ be a maximal $\delta \eta$-packing\footnote{A subset $X$ of a metric space $\calX$ is an $r$-packing if for any $x,x' \in X$, $d(x,x') \geq r$, or equivalently, if the open balls $B(x,\frac{r}{2})$ for 
   $x \in X$ are pairwise disjoint. If $X$ is a maximal $r$-packing of $\calX$, then for any $y \in \calX$, there exists $x \in X$ such that $d(x,y) < r$ - in other words $X$ is $r$-covering.} (with regard to the distance $d_N$) 
   of $\gamma$, and let $V=\bigcup_{i \in I} B(y_i, \frac{\eta \delta}{2})$ be the union of disjoint balls. We now make use of the following fact:
   
   \begin{fact}
    Let $\gamma$ be a path (a continuous map $\gamma: [0,1] \rightarrow N$) between $y$ and $y'$ in a metric space $N$. Then any maximal $r$-packing (for the distance $d_N$) of $\gamma$ has at least 
    $\frac{d_N(y,y')}{2r}$ elements.
   \end{fact}
   
   \begin{proof}
    Let $\{y_i\}_{i \in I}$ be an $r$-packing of $\gamma$.
    Consider the graph whose vertices are the $y_i$ and where there exists an edge between $y_i$ and $y_j$ when $d_N(y_i,y_j) \leq 2 r$. This graph is connected: indeed, if $I = I_0 \sqcup I_1$,  
    where $I_0$ is the set of indices of a (non-empty) connected component of the graph, one has $d_N(y_i,y_{i'}) > 2r$ for every $i \in I_0$, $i' \in I_1$. The open subsets $V_0=\bigcup_{i \in I_0} B(y_i, r)$ and 
    $V_1=\bigcup_{i \in I_1} B(y_i, r)$ of $\calX$ are disjoint. Meanwhile $\{y_i\}_{i \in I}$ is an $r$-covering of $\gamma$, so $\gamma \subset V_0 \cup V_1$. $\gamma$ being a continuous path, it is entirely contained 
    either in $V_0$ or in $V_1$, and since $y_i \in \gamma$ for all $i \in I$, the only possibility is $I_0=I$ and $I_1=\emptyset$ ($I_0$ was supposed non-empty in the first place). Then obviously, any two points of a connected 
    graph are connected by a path whose length is at most the size of the graph. Thus $d_N(y,y') \leq 2 r |I|$.
   \end{proof}
   
   Here, since $\vol{} V \geq \sum_{i \in I} \vol{} B(y_i, \frac{\delta \eta}{2})$, by using the fact above as well as assumption $(f)$, one finds
   \begin{equation} \label{volume_V}
    \vol{} V \geq \frac{\omega_n}{2^{n+2}} ( \delta \eta )^{n-1} d_N(\psi(x),\psi(x')). 
   \end{equation}
   Meanwhile $\overline{xx'} \subset B(x,\delta)$, thus from assumption $(g)$, $\psi^{-1}(V) \subset B(x, 2\delta)$. Assumption $(h)$ then implies
   $\vol{} V \leq A \vol{} B(x,\delta) \leq c(n) \delta^n$. Comparing with $(\ref{volume_V})$, one finds
   \[ d_N(\psi(x),\psi(x')) \leq  \frac{c(n)A}{\eta^{n-1}} \delta. \]
   Finally, considering $x,x' \in B(x_0,\frac{1-\delta}{2})$ and dividing a minimizing geodesic $\overline{xx'} \subset B(x_0,1-\delta)$ in at most $\frac{d_M(x,x')}{\delta}+1$ segments of length less than 
   $\delta$, one gets $(\ref{psi_haussdorf_lipschitz})$.

   \bigskip

   \item[\bf (iii)] The image of $\psi$ contains a ball of fixed size, namely $B(y_0,\frac{1}{2}-2 \delta) \subset \psi(B(x_0, \frac{1-\delta}{2})).$ Indeed, since $\psi$ is an embedding,
   \[ \psi ( \bound B(x_0, \frac{1-\delta}{2} ) ) = \bound \left( \psi(B(x_0,\frac{1-\delta}{2})) \right). \]
   Thus, if $y \in \bound \left( \psi(B(x_0,\frac{1-\delta}{2}) \right)$, $y=\psi(x)$ for some $x \in S(x_0,\frac{1-\delta}{2})$, so from assumption $(g)$, $d(y_0,y) \geq \frac{1-\delta}{2} - \delta$. 
   So if there were some $z \in N \setminus \psi(B(x_0,\frac{1-\delta}{2}))$ with $d(y_0,z) \leq \frac{1}{2}-2 \delta$, there would also exist $y \in \overline{y_0 z} \cap \bound \left( \psi(B(x_0,\frac{1-\delta}{2}) \right)$, 
   in contradiction with what we just said.

   \bigskip

   \item[\bf(iv)] We get a preliminary volume lower bound (depending on $\eta$)
   \begin{equation}
    \vol{} B(y_0,1) \geq v_1(n, \calI_0, \eta). \label{minoration_v1}
   \end{equation}
   Indeed, assuming $\delta < \frac{1}{2 \lambda}$, one gets $\psi(B(x_0,\frac{1}{2 \lambda})) \subset B(y_0,1)$ from $(\ref{psi_haussdorf_lipschitz})$. Now if $\{x_i\}_{i \in I}$ is a maximal $2 \delta$-packing of 
   $B(x_0,\frac{1}{2 \lambda})$, $\{x_i\}_{i \in I}$ is also $2 \delta$-covering, so we have on the one hand $\vol{} B(x_0,\frac{1}{2 \lambda}) \leq \sum_{i \in I} \vol{} B(x_i,\delta)$, 
   and thus, from $(\ref{minoration_v0})$,
   \[ \frac{v_0(n, \calI_0)}{(2 \lambda)^n}  \leq C(n) \delta^n |I|. \]
   On the other hand, by assumption $(g)$, the $y_i=\psi(x_i)$ for $i \in I$ form a $\delta$-packing of $B(y_0,1)$. In particular, thanks to assumption $(f)$,
   \[ \vol{} B(y_0,1) \geq c(n) |I| (\eta \delta)^n, \]
   so finally $\vol{} B(y_0,1) \geq c(n) \frac{ v_0(n, \calI_0) \eta^n}{\lambda^n}$.
     
  \end{itemize}

  \bigskip

  The proof of the lemma is by contradiction. One fixes $\calI_0, A, \eta>0$ and considers a sequence $\delta_k \rightarrow 0$, as well as sequences $(M_k,x_{0,k})$,$(N_k,y_{0,k})$, and $\psi_k$ verifying 
  hypothesis $(a)$ through $(i)$.

  \bigskip

  \begin{itemize}
   \item[\bf(v)] We first extract limit spaces $(X,\bx_0)$, $(Y,\by_0)$, and a map $\bpsi: X_0 \subset X \rightarrow Y$ from the sequence of the $\psi_k: M_k \rightarrow N_k$. By Gromov's (relative) compactness theorem for the 
   space of manifolds with Ricci curvature bounded from below, there exist complete metric spaces $(X,\bx_0)$ and $(Y,\by_0)$ such that, after extraction,
   \begin{align*}
    & \left( \overline{ B \left( x_{0,k}, \frac{1-\delta_k}{2} \right) } , x_{0,k} \right)  \underset{G-H}{\longrightarrow} (X,\bx_0),\\
    & \left( \overline{ B \left( y_{0,k}, \frac{1}{2} - 2 \delta_k \right) } , y_{0,k} \right)  \underset{G-H}{\longrightarrow} (Y,\by_0).
   \end{align*}
   
   Moreover, since in both cases the limit is non-collapsed, (by $(\ref{minoration_v0})$ and $(\ref{minoration_v1})$) results from Cheeger-Colding's theory (see \cite{Cheeger_2001}) guaranties that the sequences of Riemannian 
   measures $\dvol{g_k(x)}$ on $B(x_{0,k},\frac{1-\delta_k}{2})\subset M_k$ and $\dvol{h_k(x)}$ on $B(y_{0,k},\frac{1}{2}-2 \delta_k) \subset N_k$ converge toward the $n$-dimensional Hausdorff measure on $\calH^n_X$ and 
   $\calH^n_Y$ respectively.

   \bigskip

   Since by $(iii)$, $B(y_{0,k},\frac{1}{2}-2 \delta_k)$ is contained in the image of the map $\psi_k$, we first construct an application $\bPsi: Y \rightarrow X$ as the limit of the $\psi_k^{-1}$. To do this consider a dense 
   subset $\calY=\{\by_i\}_{i \in I}$ of $Y$, their lifts $y_{i,k} \in B(y_{0,k},\frac{1}{2}-2 \delta_k)$ with $y_{i,k} \rightarrow \by_i$, as well as their preimages $x_{i,k}=\psi_k^{-1}(y_{i,k})$ in $M_k$. By assumption $(g)$,
   $x_{i,k} \in B(x_{0,k},\frac{1}{2}-\delta_k)$. Thus we assume, after extraction, that the $x_{i,k}$ converge toward limits $\bx_i \in X$. We define the map $\bPsi: \calY \rightarrow X$ by setting $\bPsi(\by_i)=\bx_i$. The 
   property that   
   \[ d_X(\bPsi(\by_i),\bPsi(\by_{i'})) \leq d_Y(\bx_i,\bx_{i'}) \]
   for all $i,i' \in I$ is immediate. The map $\bPsi$ being in particular uniformly continuous, it extends to a map defined on the complete space $Y$, with   
   \[ \frac{1}{\lambda} d_Y(y,y') \leq d_X(\bPsi(y),\bPsi(y')) \leq d_Y(y,y') \]
   for every $y,y' \in Y$, the lower bound being got by taking the limit in $(\ref{psi_haussdorf_lipschitz})$. In particular, $\bPsi$ is a homeomorphism on $X_0:=\bPsi(Y)$, and its inverse $\bpsi:=\bPsi^{-1}: X_0 \rightarrow Y$
   satisfies the following properties
   \begin{align*}
    & d_X(x,x') \leq d_Y(\bpsi(x),\bpsi(x')) \leq \lambda d_X(x,x') \text{ for all } x,x' \in X_0, \\
    & \calH^n_Y \bpsi(U) \leq A \calH^n_X U \text{ for any } U \subset X_0.
   \end{align*}
   Moreover, $\psi$ being non-contracting, an elementary property of Hausdorff measures yields
   \[ \calH^m_Y \bpsi(U) \geq \calH^m_X U \text{ for any } U \subset X_0, m \in \setN. \]

   \bigskip

   \item[\bf(vi)] The existence of the map $\bpsi$ with the properties listed above implies a control on the isoperimetric ratio of balls centered at $\by_0$ in $Y$. Let $0<r \leq \frac{1}{2}$ be fixed and consider
   \[ \begin{array}{cc} B = B(\by_0,r), & S = S(\by_0,r) \end{array}, \]
   as well as the corresponding sets $B_k=B(y_{0,k},r)$, $S_k=S(y_{0,k},r)$ in $N_k$. Consider also the preimages $\tS \subset X_0$ by $\bpsi$ and $\tB_k, \tS_k \subset M_k$ by $\psi_k$.

   \bigskip

   In order to establish an upper bound on the measure of $B$, first recall that $\calH_Y^n B = \lim_k \vol{n} B_k$, and $\vol{n} B_k \leq A \vol{n} \tB_k$ by assumption $(h)$. Now, by definition of Hausdorff measure, there exists 
   a finite set of points $\{\bx_i\}_{i \in I}$ of $X_0$ and radii $\{\eta_i\}_{i \in I}$ such that the $B(\bx_i,\eta_i)$ cover $\tS$, and
   \[ \sum_{i \in I} \omega_{n-1} \eta_i^{n-1} \leq 2 \calH^{n-1}_X \tS. \] 
   Let $\eta = \min_{i \in I} \eta_i$. For $k$ large enough, the $\bx_i$ can be lifted to points $x_{i,k} \in M_k$ such that
   \[ \tS_k \subset \bigcup_{i \in I} B(x_{i,k},2 \eta_i). \]
   If $U_k = \tB_k \cup \bigcup_{i \in I} B(x_{i,k},\eta_i)$, one has $\bound U_k \subset \bigcup_{i \in I} S(x_{i,k},\eta_i)$, whence
   \[ \vol{n-1} \bound U \leq \sum_{i \in I} c(n) \eta_i^{n-1} \leq c(n) \calH_X^{n-1} \tS. \]
   Making use of assumption $(c)$ and the fact that $\tB_k \subset U_k$, one has
   \[ \vol{n} \tB_k \leq c(n) \calI_0 \left( \calH_X^n \tS \right)^\frac{n}{n-1}. \]   
   Finally since $\calH_X^{n-1} \tS \leq \calH_Y^{n-1} S$, one finds
   \begin{equation}
    \calH_Y^n B(\by_0,r) \leq \calI(n,\calI_0,A) \left( \calH_Y^{n-1} S(\by_0,r) \right)^\frac{n}{n-1}. \label{majoration_I}
   \end{equation}

   \bigskip

   \item[\bf(vii)] Conclusion. Define $f(r) := \calH_Y^n B(\by_0,r)$. Inequality $(\ref{majoration_I})$ becomes 
   \[ f(r) \leq \calI(n,\calI_0,A) \left( f'(r) \right)^\frac{n}{n-1} \]
   which can be integrated into $f(r) \geq v(n,\calI_0,A) r^n$. For $k$ large enough one thus has
   \[ \vol{h} B(y_{0,k},r) \geq v(n,\calI_0,A) r^n \]
   for $\eta \leq r \leq 1$. That the inequality still holds for $0 < r < \eta$ is then a direct consequence of assumption $(e)$. 
   
  \end{itemize}

 \end{proof}
 
 For later use we write the lemma above under a slightly different form. This reformulation, obtained by mere scale manipulations, says that for a given $\delta$, a $\delta$-non contracting map won't collapse balls of
 radius $\Lambda \delta$, for some $\Lambda$ large enough, depending on the hypothesis. 
 
 \begin{corollary}[Reformulation] \label{reformulation}
  There exist $v(n,v_0,A)>0$ and $\Lambda_0(n,v_0,A,\eta)>0$ such that the following holds.
  Let $(M,g,x_0)$ be an $n$-dimensional Riemannian manifold such that for some $\Lambda \geq \Lambda_0$,
  \[ \begin{array}{ll}
   (a) & B(x_0, 2 \Lambda \delta) \text{ is relatively compact in $M$,} \\
   (b) & \ricci g \geq -\frac{1}{\Lambda^2 \delta^2 \eta^2} \text{ on $B(x_0,2 \Lambda \delta)$, } \\
   (c) & \vol{n} B(x_0, r) \geq v_0 r^n \text{ for } 0<r<2 \Lambda \delta. 
  \end{array} \]
  Let $(N,h,y_0)$ be an $n$-dimensional Riemannian manifold such that 
  \[ \begin{array}{ll}
   (d) & B(y_0, \Lambda \delta) \text{ is relatively compact in $N$,} \\
   (e) & \ricci h \geq -\frac{1}{\Lambda^2 \delta^2 \eta^2} \text{ on } B(y_0,\Lambda \delta),  \\
   (f) & \vol{h} B(y,r) \geq \frac{\omega_n}{2} r^n \text{ for } y \in N, 0 < r < \eta \delta.
  \end{array} \]
  Let finally $\psi: (M,x_0) \rightarrow (N,y_0)$ be an embedding such that 
  \[ \begin{array}{ll}
   (g) & d_N(\psi(x),\psi(x')) \geq d_M(x,x') - \delta \text{ for every $x,x' \in M$,} \\
   (h) & \vol{h} \psi(U) \leq A \vol{g} U \text{ for any $U \subset M$, } \\
  \end{array} \] then
  \[ \vol{h} B(y_0,r) \geq v r^n \text{ for } 0 < r \leq \Lambda \delta, \]
  moreover $d_N(\psi(x),\psi(x')) \leq \frac{c(n) \Lambda \delta A}{\eta^{n-1}} \left( d_M(x,x') + \delta \right)$. 
 \end{corollary}

\section{Proof of proposition $\ref{main_proposition}$}

In this section we prove proposition $\ref{main_proposition}$, which we decomposes into two propositions. We prove first that a flow with an a priori $\frac{K}{t}$ curvature bound is non-collapsed for a time 
$\epsilon_{\ref{main_I}}^2$ and displays a pinching of the Ricci curvature of the form $-\frac{1}{t}$. While $\epsilon_{\ref{main_I}}$ depends on $K$, it is essential that the non-collapsing constant and the value of the Ricci 
lower bound does not. 

 \begin{proposition} \label{main_I}
 
  There exist $v(v_0)$ and $\epsilon_{\ref{main_I}}(v_0,K)$ such that the following holds.
  Let $g(t)$ be a Ricci flow on $(M^3,x_0) \times [0,T[$ such that $B_t(x_0,1)$ is relatively compact in $M$ for all $0 \leq t<T$. Suppose moreover that
  
  \[ \begin{array}{ll}
      (a) & \ricci g(0) \geq -1 \text{ on $B_0(x_0,1)$, } \\
      (b) & \vol{} B_0(x_0,r) \geq v_0 r^3 \text{ for } 0 < r \leq 1, 
     \end{array} \]
  as well as  
  \[ \begin{array}{ll}
      (c) & \normes{ \riem g(x,t) } \leq \frac{K}{t}, \text{ for $0 \leq t \leq T$, $x \in B_t(x_0,1)$, } \\
      (d) & \inj{g(t)} x \geq \sqrt{\frac{t}{K}} \text{ for } x \in B_t(x_0,1).
     \end{array} \]
  Then   
  \begin{align*}
   & \vol{} B_t(x_0, r) \geq v r^3, \\
   & \ricci g(x,t) \geq -\frac{1}{t},
  \end{align*}
  for $0 < r \leq \sqrt{t}$, $x\in B_t(x_0,\sqrt{t})$ and $0 \leq t \leq \min(\epsilon_{\ref{main_I}}^2,T)$. 

 \end{proposition}
 
  \begin{proof}
 
  \begin{itemize}
  
   \item[\bf(i)] Let $0<t<T$. To establish estimates on $g$ at time $t$, we consider the rescaled flow $\tg(s) = \frac{1}{t}g(t s)$ defined on $M$ for $0 \leq s \leq 1$. The flow $\tg(s)$ satisfies the following properties at 
   initial time:     
    \begin{equation} \label{hyp_rescale_0}
     \begin{array}{cc}
      \ricci \tg(x,0) \geq -t, & \vol{} B_{\tg(0)}(x_0,r) \geq v_0 r^3, 
     \end{array}      
    \end{equation}
    for all $x \in B_{\tg(0)}(x_0,\frac{1}{\sqrt{t}})$, $0< r \leq \frac{1}{\sqrt{t}}$, as well as    
    \begin{equation} \label{hyp_rescale_1}
     \begin{array}{cc} \normes{\riem \tg(x,s)} \leq \frac{K}{s}, & \inj{\tg(s)} x \geq \sqrt{\frac{s}{K}}, \end{array}
    \end{equation}
    for $0 < s \leq 1$, $x \in B_{\tg(s)}(x_0, \frac{1}{\sqrt{t}})$.

    \bigskip
    
    Let us note first that $(\ref{hyp_rescale_1})$ implies the existence of a regularity scale $r_{reg}$ at time $s=1$, at which balls have at least half the Euclidian volume. Actually $r_{reg}:=\frac{1}{100 \sqrt{K}}$ works, i.e.
    \begin{equation}
     \vol{} B_{\tg(1)}(x,r) \geq \frac{2 \pi}{3} r^3 \text{ for } 0 < r < \frac{1}{100 \sqrt{K}}, \label{regularite}
    \end{equation}
    for every $x \in B_{\tg(1)}(x_0, \frac{1}{\sqrt{t}})$, while on the other hand we have, by corollary $\ref{corollaire_distorsion_distances}$    
    \begin{equation} \label{distortion}
     d_1(x,x') \geq d_0(x,x') - 20 \sqrt{K}
    \end{equation}
    for all $x,x' \in B_{\tg(0)}(x_0, \frac{1}{3 \sqrt{t}})$.
    
    \bigskip
    
    We now specify the parameters with which we wish to apply corollary $\ref{reformulation}$, in order to compare the metrics $\tg(0)$ and $\tg(1)$. The identity map between both metrics has distance contraction controlled by
    \[ \delta = 20 \sqrt{K}, \]
    while $\tg(1)$ is regular at scale $r_{reg} = \frac{1}{100 \sqrt{K}}$, whence the choice of  
    \[ \eta = \frac{r_{reg}}{\delta} = \frac{1}{2000 K}. \]
    so as to have $r_{reg}=\delta \eta$. The corollary then yields a $\Lambda_0(3,v_0,2,\eta)$ and we define $\Lambda = \max(\Lambda_0, 100\sqrt{K})$ for a reason which will become clear later. Thus if $t$ is
    small enough so as to have    
   \begin{equation}\label{condition_t1}
    \frac{1}{3 \sqrt{t}} \geq 3 \Lambda \delta,
   \end{equation}
   then $(\ref{regularite})$ and $(\ref{distortion})$ hold for all $x,x' \in B_{\tg(1)}(x_0, 3 \Lambda \delta)$ guarantying thereby that the identity map satisfies hypothesis $(f)$ et $(g)$ of the corollary.

   \bigskip
    
   \item[\bf(ii)] We now need to establish the following lower bound on the Ricci curvature of $\tg(1)$, in order to match hypothesis $(e)$   
   \begin{equation}   
    \begin{array}{cc} \ricci \tg(x,1) \geq -\frac{1}{\Lambda^2 \eta^2 \delta^2} = - \frac{10^4 K}{\Lambda^2} & \text{ for } x \in B_{\tg(1)}(x_0, 3 \Lambda \delta) \end{array}. \label{hypothese_e}
   \end{equation}
   To do this we apply corollary $\ref{ricci_local}$ choosing $\alpha = \frac{10^4 K}{\Lambda^2}$ and $a = \frac{2}{3 \sqrt{t}}$. We get $\epsilon_1=\epsilon_1(\alpha,K)$ such that if   
   \begin{equation} \label{condition_t2}
    \sqrt{t} < \frac{\epsilon_1}{3},
   \end{equation}
   then $B_{\tg(0)}(x_0,a+\frac{1}{\epsilon_1}) \subset B_{\tg(0)}(x_0, \frac{1}{\sqrt{t}})$, so with $(\ref{hyp_rescale_0})$ one finds   
   \[ \ricci \tg(x,s) \geq -\frac{\alpha}{s} \]
   for every $x \in B_{\tg(s)}(x_0,\frac{2}{3 \sqrt{t}}-20 \sqrt{K})$, $0 \leq s \leq 1$. Thanks to $(\ref{condition_t1})$, this implies $(\ref{hypothese_e})$.
   
   \bigskip
   
   \item[\bf(iii)] Finally we control the dilatation of regions. Let us consider $U \subset B_{\tg(0)}(x_0, 3 \Lambda \delta)$, and $x_1 \in U$. We apply $\ref{proposition_minoration_scalaire}$, with 
   $a=A=\frac{1}{4 \sqrt{t}}$. Then $B_{\tg(0)}(x_1, a+A) \subset B_{\tg(0)}(x_0, \frac{1}{\sqrt{t}})$ by $(\ref{condition_t1})$. Since $\scal \tg(0) \geq - 3 t$ (by $(\ref{hyp_rescale_0})$) on this ball,
   we get, for $0 \leq s \leq 1$, $\scal \tg(s) \geq -1600 t$ on $B_{\tg(s)}(x_1, \frac{1}{4 \sqrt{t}} - 20 \sqrt{K s})$. In particular (recall $(\ref{condition_t1})$), 
   \[ \scal \tg(x_1,s) \geq -1600 t \]
   for all $x_1 \in U$, $0 \leq s \leq 1$. Thus we have in particular
   \[ \vol{\tg(1)} U \leq e^{1600 t} \vol{\tg(0)} U \leq 2 \vol{\tg(0)} U \]
   assuming for example
   \begin{equation} \label{condition_t3}
    t < \frac{1}{3200}.
   \end{equation}
   
   \bigskip
   
   \item[\bf(iv)] The $\epsilon_{\ref{main_I}}$ of the proposition is obtained by collecting the different upper bounds $(\ref{condition_t1})$, $(\ref{condition_t2})$ and $(\ref{condition_t3})$ on $\sqrt{t}$ which have appeared 
   in the course of the argument. Then if $t < \epsilon_{\ref{main_I}}^2$, corollary $\ref{reformulation}$ applies to the map
   \[ \begin{array}{cc} \psi : & (M,\tg(0),x_0) \rightarrow (N,\tg(1),x_0) \end{array} \]
   with $M=N=B_{\tg(0)}(x_0, 3 \Lambda \delta)$, and $\psi \equiv \id$. One finds that
   \[ \vol{\tg(1)} B_{\tg(1)}(x_0,r) \geq v(3,v_0,2) r^3, \]
   for $0 < r < \Lambda \delta$, or, coming back to $g$, 
   \[ \vol{g(t)} B_t(x_0,r) \geq v r^3 \]
   for any $0 < r < \Lambda \delta \sqrt{t}$, so (by the choice we made for $\Lambda$) in particular for $0 < r < \sqrt{t}$, and for all $0< t < \epsilon_{\ref{main_I}}^2$. Finally, since we also have $\frac{10^4 K}{\Lambda^2} 
   \leq 1$, $(\ref{hypothese_e})$ implies, when coming back to the metric $g$, 
   \[ \ricci g(x,t) \geq -\frac{1}{t} \]
   for $x \in B_t(x, \sqrt{t})$ which is the second conclusion of the proposition.
    
  \end{itemize}

 \end{proof}
 
 The next proposition is classical and formalizes the fact that in dimension $3$, a metric evolving by the Ricci flow becomes collapsed around points at which the curvature becomes high. In its local version, it is due to Miles 
 Simon (see for example thm 2.1 in \cite{Simon_2013}). We nevertheless include the proof of the precise version we will use.
 
 \begin{proposition} \label{main_II}
 
  There exists $K(v)$ such that if $g(t)$ is a (not necessarily complete) Ricci flow on $(M^3,x_0) \times [0,1]$ such that 
  \begin{align*}
   & B_t(x_0,1) \text{ is relatively compact in } M, \\
   & \vol{t} B_t(x,r) \geq v r^n \text{ whenever } B_t(x,r) \subset B_t(x_0,1) \text{ and } 0 < r \leq \sqrt{t},
  \intertext{for all $0 \leq t \leq 1$, then}
   & \normes{ \riem g(x,t) } \leq \frac{K}{t}
  \end{align*}
  for $x \in B_t(x_0, \frac{1}{2})$ and $0 \leq t \leq 1$.
  
 \end{proposition}
 
 A key ingredient when proving this kind of local result is a point-picking technique introduced by Perelman (see \cite{Perelman_2002}, 10.1 or \cite{Kleiner_Lott}, 31.1 and 32.2), which we use under the following form
 
 \begin{lemma} \label{point_picking}
  Let $g(t)$ be a (not necessarily complete) Ricci flow on $M \times [0,1]$ such that $B_t(x_0,1)$ is relatively compact in $M$ for $0 \leq t \leq 1$. If the space-time point $(x,t) \in M \times [0,1]$ is such that   
  \begin{align*}
   & x \in B_t(x_0,\frac{1}{2}), \\
   & \normes{\riem g(x,t)} = \frac{Q}{t},
  \intertext{then for all $A \leq \frac{\sqrt{Q}}{8}$, there exists another space-time point $(\bx,\bt)$ with the following properties}
   & \bt \leq t, \\
   & d_{\bt}(x_0,\bx) < d_t(x_0,x) + \frac{2A}{\sqrt{Q}}, \\
   & \bQ \geq Q, 
  \intertext{where $\bQ$ is defined by $\normes{\riem g(\bx,\bt)}=\frac{\bQ}{\bt}$, and}
   & \normes{ \riem g(x',t')} \leq \frac{4 \bQ}{t'}, 
  \end{align*}
  for all $(x',t')$ such that $0 \leq t' \leq \bt$ and $d_{t'}(x_0,x') < d_{\bt} (x_0,\bx) + A \sqrt{\frac{\bt}{\bQ}}$. This upper bound holds in particular on the parabolic domain
  $B_{\bt}(\bx, \frac{A}{10} \sqrt{\frac{\bt}{\bQ}}) \times [\bt (1 - \frac{A}{\bQ}),\bt]$, which is contained in $\cup_{t=0}^1 B_t(x_0,1) \times \{ t \} $.
 \end{lemma}

 \begin{proof}[Proof of proposition $\ref{main_II}$]
  Let us consider a sequence of Ricci flows $g_i(t)$ defined for $0 \leq t \leq 1$ on $3$-dimensional manifolds $(M_i,x_{0,i})$ and contradicting the conclusion of the proposition, which means that on one side
  \[ \vol{g_i(t)} B_t(x,r) \geq v r^3 \]
  for all $(x,t) \in M_i \times [0,1]$ with $d_t(x_{0,i},x) < \sqrt{t}$ and for all $r \leq \sqrt{t}$, while on the other side there exists points $(x_i,t_i) \in M_i \times [0,1]$ with $d_{t_i}(x_{0,i},x_i) \leq \frac{1}{2}$ and 
  \[ \normes{\riem g_i(x_i,t_i)} > \frac{K_i}{t_i}. \]  
  Lemma $\ref{point_picking}$, applied with $A_i=\frac{\sqrt{K_i}}{8} $ allows us to choose a new sequence of points $(\bx_i,\bt_i) \in M_i \times [0,1] $ with $\bt_i \leq t_i$ and $\bx_i \in B_{\bt_i}(x_{0,i},\frac{3}{4})$ 
  such that, for $\bK_i=\bt_i \normes{ \riem g_i(\bx_i,\bt_i) }$, one has  
  \[ \begin{array}{cc} \bK_i \geq K_i, & \normes{\riem g_i(x,t)} \leq \frac{4 \bK_i}{t} \end{array} \]
  for $x \in B_{\bt_i}(\bx_i, \frac{\sqrt{K_i}}{80} \sqrt{\frac{\bt_i}{\bK_i}}) \times [\bt_i-\frac{\sqrt{K_i}}{8} \frac{\bt_i}{\bK_i}, \bt_i]$. One considers then the sequence  
  \[ h_i(x,s) = \frac{\bK_i}{\bt_i} g_i ( x, \bt_i + \frac{s \bt_i}{\bK_i} ) \]
  of normalized flows (by which we mean that $\normes{ \riem h_i(\bx_i,0) } = 4$) on $M_i \times [-\frac{\sqrt{K_i}}{8},0]$. This sequence satisfies, for all $r>0$,  
  \begin{align}
   \normes{\riem h_i(x,s)} \leq \frac{4}{1+\frac{s}{\bK_i}} \leq 8 \label{borne_courbure}
  \end{align}
  on $B_{h_i(0)}(\bx_i,r)$ as soon as $\frac{\sqrt{K_i}}{80} > r$, as well as  
  \begin{align}
   \vol{} B_{h_i(0)}(\bx_i,r) \geq v r^3 \label{non_effondrement}
  \end{align}
  as soon as $\sqrt{K_i}>r$. By Cheeger-Gromov-Taylor classical result\footnote{\label{Cheeger_Gromov_Taylor}Recall that Cheeger-Gromov-Taylor's theorem says that there exists a function $\iota(n,v)>0$ such that if a ball $B=B_g(x,r)$ is relatively compact 
  in a Riemannian manifold and satisfies $\vol{g} B \geq v r^n$ and $\normes{ \riem g } \leq \frac{1}{r^2}$ on $B$, then $\inj{g} x \geq \iota r$.}, $(\ref{borne_courbure})$ together with $(\ref{non_effondrement})$ implies a 
  lower bound on the injectivity radius at time $0$
  \[ \inj{h_i(0)} \bx_i \geq \frac{\iota(3,v)}{\sqrt{2}}. \]
  This, in conjunction with the curvature bound $(\ref{borne_courbure})$ and the fact that $K_i \rightarrow + \infty$, allows us to apply Hamilton compactness theorem to the sequence of Ricci flows $h_i$ centered at $(\bx_i,0)$ 
  and extract a limit which is a complete ancient flow $h(s)$ on a differential manifold $(M,\bx)$.
  
  \bigskip
  
  Like any ancient $3$-dimensional solution to the Ricci flow, $h(s)$ has non-negative sectional curvature (\cite{Chow_Lu_Ni}, 6.50 and \cite{Chen_2013}), while on the other hand $(\ref{non_effondrement})$ implies that for all 
  $r>0$, $\vol{h(0)} B_0(\bx,r) \geq v r^3$, that is to say that the Asymptotic Volume Ratio (AVR, see note page \pageref{page_AVR}) of $h(0)$ is strictly positive. According to Perelman (\cite{Perelman_2002}, 11.4), the only possibility is that 
  $(M,h(s))$ is the trivial static solution given by the Euclidian metric $(\setR^3, g_{\setE^3})$. But the normalizing condition $\normes{ \riem h_i(\bx_i,0) } = 4$ implies that $\normes{ \riem h(\bx,0) } = 4$ in the limit, 
  which is a contradiction.
 \end{proof}
 
 Now we combine propositions $\ref{main_I}$ and $\ref{main_II}$ to get proposition $\ref{main_proposition}$ as announced.
 
 \begin{proof}[Proof of proposition $\ref{main_proposition}$.]
 
  Let us consider a flow $g(t)$ satisfying the hypothesis of the proposition, fix some $\epsilon>0$ and some $0 < t \leq \epsilon^2$, and define the rescaled flow $\tg(s) = \frac{\epsilon^2}{t} g(\frac{t}{\epsilon^2} s)$ for
  $0 \leq s \leq \epsilon^2$. Clearly the initial data satisfies
  \[ \begin{array}{cc} \ricci \tg(x,0) \geq -1, & \vol{} \tB_0(x,r) \geq v_0 r^3 \end{array} \]
  for $x \in U$, $0 \leq r \leq 1$, moreover
  \[ \begin{array}{cc} \normes{ \riem \tg(x,s) } \leq \frac{K_1}{s}, & \inj{\tg(s)} x \geq \sqrt{\frac{s}{K_1}} \end{array} \]
  for $x \in U$ and $0 \leq s \leq \epsilon^2$. Now define $U' = \{ x \in U \ | \ \tB_0(x,2) \rcsubset U \}$\footnote{Recall that $\rcsubset$ stands for ``is relatively compact in''.}, and pick $x \in U'$. Since by lemma
  $\ref{lemme_distorsion_distances}$, 
  \[ \td_s(x,M \setminus U) \geq \td_0(x,M \setminus U) - \frac{40}{3} \sqrt{K_1} \epsilon \]
  one has $\tB_s(x,1) \rcsubset U$ for $0 \leq s \leq \epsilon^2$, as long as we choose $\epsilon \leq \frac{3}{40 \sqrt{K_1}}$. Thus one can apply proposition $\ref{main_I}$ at each point $x \in U'$ to get,
  for $v=v(v_0)$ and $\epsilon_{\ref{main_I}}=\epsilon_{\ref{main_I}}(v_0,K_1)$ as given in the proposition
  \begin{equation} \label{proof_mp_1}
   \begin{array}{cc} \ricci \tg(x,s) \geq -\frac{1}{s}, & \vol{} \tB_s(x,r) \geq v r^3 \end{array},  
  \end{equation}
  for $x \in U'$, $0 \leq r \leq \sqrt{s}$ and $0 \leq s \leq \epsilon^2$, as long as $\epsilon \leq \epsilon_{\ref{main_I}}$. Now define $U''= \{ x \in U \ | \ \tB_0(x,4) \rcsubset U \}$, and pick $x \in U$. By the same argument 
  as above, one has that for every $0 \leq s \leq \epsilon^2$, $\tB_s(x,1) \rcsubset U'$, and thus one can apply proposition $\ref{main_II}$ at each point of $U''$, wherefrom we get $K=K(v)$ such that
  \begin{equation} \label{proof_mp_2}
   \normes{ \riem \tg(x,s) } \leq \frac{K}{s} 
  \end{equation}
  for $x \in U''$, $0 \leq s \leq \epsilon^2$. Thus the $\epsilon$ of the proposition is obtained by setting $\epsilon = \min( \frac{3}{40 \sqrt{K_1}} , \epsilon_{\ref{main_I}})$, and coming back to the original metric $g$, we
  have the analogue of $(\ref{proof_mp_1})$ and $(\ref{proof_mp_2})$ for $g$ at time each time $0 \leq t \leq \epsilon^2$ on
  \[ U'' = \{ x \in U \ | \ B_0(x,\frac{4}{\epsilon} \sqrt{t}) \rcsubset U \}. \]
  
 \end{proof}
 
 As a first consequence of proposition $\ref{main_proposition}$ we prove the pseudo-locality type statement theorem $\ref{three_dim_pl}$. We will not use this result in the sequel.
 
 \begin{proof}[Proof of theorem $\ref{three_dim_pl}$]
  
  Set $K=K(v_0)$ and $v=v(v_0)$ given by proposition $\ref{main_proposition}$. Let $g(t)$ be a flow satisfying the hypothesis of the theorem.
  
  \begin{itemize}
   \item[\bf (i)]  Suppose that for some time $0 < t \leq \epsilon^2$ ($\epsilon$ to be determined) and some subset $V \subset U$, we have
    \begin{equation} \label{proof_pl_1}  
     \begin{array}{ccc} \normes{ \riem g(x,s) } \leq \frac{K}{s}, & \ricci g(x,s) \geq -\frac{1}{s}, & \vol{} B_{s}(x,r) \geq v r^3, \end{array} 
    \end{equation}  
    for all $0 \leq r \leq \sqrt{s}$, $0 < s \leq t$ and $x \in V$. By Cheeger-Gromov-Taylor, this implies $\inj{g(s)} x \geq \iota(3,v) \sqrt{\frac{s}{K}}$ for every $x$ such that $B_{s}(x,\sqrt{\frac{s}{K} }) \subset V$.
  
    \bigskip
  
    Pseudo-locality guaranties an additional time of controlled curvature and injectivity radius on a slightly smaller region. Indeed, if one sets $K_1 = \frac{16 K}{\iota(3,v)^2}$ (thus $K_1 \geq K$), then  
    \[ \begin{array}{cc} \normes{ \riem g(x,t) } \leq \frac{K_1}{16 t}, & \inj{g(t)} x \geq 4 \sqrt{ \frac{t}{K_1} }, \end{array} \]
    for every $x$ such that $B_t(x,\sqrt{\frac{t}{K}}) \subset V$. Hence we can apply corollary $\ref{misc_2}$ between time $t$ and $t'=t+\epsilon_{\ref{misc_2}}^2 \frac{t}{K_1}$ to get
    \[ \begin{array}{cc} \normes{ \riem g(x,s) } \leq \frac{K_1}{4 t}, & \inj{g(s)} x \geq 2 \sqrt{\frac{t}{K_1}}, \end{array} \]
    for every $x$ such that $B_t(x, 5 \sqrt{\frac{t}{K}}) \subset V$ and $t \leq s \leq t'$. In particular, using that by lemma $\ref{lemme_distorsion_distances}$, $d_t(x,M \setminus V) \geq d_0(x, M \setminus V) - 
    \frac{40}{3} \sqrt{K t}$, we have
    \[ \begin{array}{cc} \normes{ \riem g(x,s) } \leq \frac{K_1}{s}, & \inj{g(s)} x \geq \sqrt{\frac{s}{K_1}}, \end{array} \]
    for $0 < s \leq t'$ and $x \in V'$, where we have set $V' = \{ x \in V \ | \ B_0 \left( x, (\frac{40}{3} \sqrt{K} + \frac{5}{\sqrt{K}}) \sqrt{t} \right) \subset V \}$. Finally we apply proposition $\ref{main_proposition}$ on 
    $V'$ and we conclude that as long as $t' \leq \epsilon_{\ref{main_proposition}}^2$, the estimates $(\ref{proof_pl_1})$ continue to hold for $0 < s \leq t'$ on the region $(V)_{a' \sqrt{t'} } = \{ x \in V \ | \ B_0(x, a' \sqrt{t'}) 
    \subset V \}$, where we have set $a'=\frac{40}{3} \sqrt{K} + \frac{5}{\sqrt{K}} + a$, $a=a(v_0,K_1)$ coming from proposition $\ref{main_proposition}$.
    
    \bigskip
    
    \item[\bf (ii)] For any compact $K \subset M$ define $V_0 = U \cap K$. Merely by continuity, there exist some $t_0>0$ small enough so that the estimates $(\ref{proof_pl_1})$ hold on $V_0$ for $0 < s \leq t_0$. But then by 
    step ${\bf (i)}$ these estimates continue to hold on $V_1 = (V_0)_{a' \sqrt{t_1}}$ for $0 < s \leq t_1$, where $t_1 = (1+\frac{\epsilon_{\ref{misc_2}}^2}{K_1}) t_0$. Inductively, if we define a sequence 
    $t_i = \left( 1 + \frac{\epsilon_{\ref{misc_2}}^2}{K_1} \right)^i t_0$ and a sequence $V_i$ by setting $V_{i+1} = (V_i)_{a' \sqrt{t_{i+1}}}$, we get that $(\ref{proof_pl_1})$ hold on $V_i$ for $0 < s \leq t_i$,
    as long as $t_i \leq \epsilon_{\ref{main_proposition}}^2$. A simple computation yields
    \[ a' \sum_{j=1}^i \sqrt{t_j} \leq \frac{10 K_1 a'}{\epsilon_{\ref{misc_2}}^2} \sqrt{t_i}. \]
    Thus if we choose $\epsilon \leq \left( 1 + \frac{\epsilon_{\ref{misc_2}}^2}{K_1} \right)^{-\frac{1}{2}} \epsilon_{\ref{main_proposition}}$ (which defines the $\epsilon_{\ref{three_dim_pl}}$ of the theorem) and 
    pursue the construction up to the first step $i$ such that $t_i \geq \epsilon^2$, we get the desired estimates for $0 \leq t \leq \epsilon^2$, on $V_i \supset (V_0)_{A \epsilon}$ for 
    $A = \frac{20 K_1 a'}{\epsilon_{\ref{misc_2}}^2}$. Since the argument works for arbitrary choice of the compact $K$, it is easy to see that the conclusion holds on $(U)_{A \epsilon}$.
   
  \end{itemize}

 \end{proof}

\section{Construction of a ``partial flow''} \label{section_partial_flow}

In this section we carry out the construction of the ``partial flow'' as defined by the statement of theorem $\ref{theorem_partial_flow}$. As we will see, the proof is by induction on an increasing sequence of times $0 < t_i 
\leq 1$, where, supposing the construction of the flow $g$ has been carried out up to time $t_i$, existence on the time interval $[t_i,t_{i+1}]$ is obtained by applying Shi's existence theorem to a modification of the metric 
$g(t_i)$. The modification essentially consists in making the metric complete by ``pushing the boundary to infinity'' by a conformal transformation while keeping track of the scale at which the geometry of the metric is controlled.

\subsection{``Pushing the boundary'' of a non-complete metric to infinity}

Let us first introduce a very elementary smoothing lemma for real valued functions on a manifold.

 \begin{lemma} \label{lemme_regularisation}
  Let $M^n$ be a Riemannian manifold, and let $f: M \rightarrow \setR$ be a $1$-Lipschitz real valued function. If $U \subset \{ x \in M | B(x,1) \text{ is relatively compact in }M \}$ is such that
  \[ \begin{array}{l}
   \normes{ \riem g} \leq 1, \\
   \inj{x} g \geq 1,
  \end{array} \]
  for all $x \in U$, then there exists a smooth function $\tf: M \rightarrow \setR$ such that
  \[ \begin{array}{l}
   \normes{f(x) - \tf(x)} \leq 1, \\
   \normes{ \nabla \tf(x) } \leq C(n), \\
   \normes{ \nabla^2 \tf(x) } \leq C(n).
  \end{array} \] 
  for all $x \in U$. Moreover, if $f \geq 0$ (resp. $f \leq 0$) on $V \subset U$ then $\tf \geq 0$ (resp. $\tf \leq 0$) on $\{ x \in V \ | \ d(x,M \setminus V) \geq 1 \}$.
  \end{lemma}
  
  From now on, $C(n)$ stands for a constant depending only on the dimension, which we allow to change from line to line.
  
  \begin{proof}
   \begin{itemize}
    \item[\bf (i)] Let $x_0 \in U$, and $\rho(x) = d(x_0,x)$. By the curvature and injectivity radius assumption, $\rho$ is a smooth function on $B(x_0,1)$, moreover, the hessian of $\rho$ is bounded from both sides on this ball:
    \begin{equation} \label{controle_hessien}
     - \rho g \leq \nabla^2 \rho \leq \frac{C(n)}{\rho} g 
    \end{equation}
    Indeed, by a classical comparison principle, hypothesis $ \riem g \geq -1$ implies
    \[ \nabla^2 \rho \leq \frac{C(n)}{\rho} g \]
    on $B(x_0,1)$. Moreover, if $\gamma: [0, \rho] \rightarrow M$ is a minimizing geodesic between $x_0$ and $x$, and $V_1 \in T_x M$, one has
    \[ \nabla^2 \rho(V_1,V_1) = \int_0^\rho \normes{\nabla_{\dot \gamma} V}^2 - K( \dot\gamma \wedge V ) \normes{V}^2 dt \]
    where $V$ is a Jacobi field along $\gamma$ with $V(0)=0$, $V(\rho)=V_1$, that can be expressed under the form $V(t) = f(t) E(t)$, where $E(t)$ is a unit norm parallel vector field along $\gamma$ and $f$ satisfies
    \[ \ddtt f = -K( \dot\gamma \wedge E ) f, \]
    with $f(0)=0$, $f(\rho)=|V_1|$. Since $K \leq 1$, we consider the comparison solution $h(t) = \frac{|V_1|}{\sin \rho} \sin t$ of equation $ \ddtt h = - h$. One then gets $\ddtt (f-h) = (1-K) (f-h)$, $f-h$ 
    cancelling at $0$ and $\rho$. Hence $f-h \leq 0$ on $[0,\rho]$ by the maximum principle, and $f \leq |V_1|$ (since $\rho \leq 1$). Thus $\nabla^2 \rho(V_1,V_1) \geq - \int_0^\rho f^2 \geq - \rho \normes{V_1}^2$ 
    whence $(\ref{controle_hessien})$.
    
    \bigskip
    
    \item[\bf (ii)] Let $ \phi: \setR_+ \rightarrow \setR_+ $ be a fixed one real variable function such that $\phi \equiv 1$ on $]-\infty, \frac{1}{2} ]$, $\phi \equiv 0$ on $]1,+\infty[$, and
    $\normes{ \phi' }, \normes{ \phi'' } \leq \calC$. For $x_0 \in U$, $\Phi_{x_0} (x) = \phi ( d(x_0,x) )$ defines a non-negative function supported in $B(x_0,1)$ with $\Phi \equiv 1$ on $B(x_0, \frac{1}{2})$, as well as
    $\normes{ \nabla \Phi_{x_0} }, \normes{ \nabla^2 \Phi_{x_0} } \leq C(n)$.
    
    \bigskip
    
    \item[\bf (iii)] We now construct a partition of the unity on $U$. Let $\{ x_i \}_{i \in I}$ a maximal $\frac{1}{2}$-packing (thus also a $\frac{1}{2}$-covering) of $U$. For all $i$, one sets 
    \[ \psi_i(x) = \left( \sum_{j \in I} \Phi_{x_j}(x) \right)^{-1} \Phi_{x_i}(x), \]
    which defines a partition of the unity subordinate to the $B(x_i,1)$. Moreover, hypothesis $ \riem g \geq -1$ guaranties that for every 
    $x \in U$, if $I(x)=\{ i \in I | x \in B(x_i,1) \}$, then $\normes{I(x)} \leq C(n)$. Finally, for every $x \in U$, there exists $i \in I$ such that $x \in B(x_i, \frac{1}{2})$ so $\sum_{j \in I} \Phi_{x_j} \geq 1$. Thus, we compute,
    at $x$,
    \begin{align*}
     \nabla \psi_i = \frac{ \nabla \Phi_{x_i} }{ \sum_{j \in I(x) } \Phi_{x_j} }- \frac{\left( \sum _{j \in I(x)} \nabla \Phi_{x_j} \right) \Phi_{x_i}}{\left( \sum _{j \in I(x)} \Phi_{x_j} \right)^2}
    \end{align*}
    whence $\normes{ \nabla \psi_i } \leq C(n)$. Similarly, one checks that $\normes{\nabla^2 \psi_i} \leq C(n)$.

    \bigskip
    
    \item[\bf (iv)] We set $\tf(x) = \sum_{i \in I} \psi_i(x) f(x_i)$, so $ \tf(x)-f(x) = \sum_{i \in I(x)} \psi_i(x) ( f(x_i) - f(x) ) $, however if $i \in I(x)$, $d(x_i,x) \leq 1$ and $|f(x_i)-f(x)| \leq 1$, whence 
    $\normes{\tf(x)-f(x)} \leq 1$. Finally, by writing $\nabla \tf = \sum_{i \in I(x)} \nabla \psi_i ( f(x_i) - f(x) )$ and $\nabla^2 \tf = \sum_{i \in I(x)} \nabla^2 \psi_i ( f(x_i) - f(x) )$, (since at every point of $U$,
    $\sum_{i} \nabla \psi_i = \sum_{i} \nabla^2 \psi_i =0$), one checks that $\normes{\nabla \tf}, \normes{\nabla^2 \tf} \leq C(n)$.
   \end{itemize}

  \end{proof}
  
  Given an open domain in a Riemannian manifold where curvature and injectivity radius is controlled, it is possible to ``push the boundary to infinity'' by a conformal modification of the metric in a neighborhood of the boundary,
  while keeping curvature and injectivity radius controlled.
  
  \begin{lemma} \label{lemme_chirurgie}
   Let $U$ be an open domain in a Riemannian manifold $(M,g)$ such that for every $x \in U$, $B(x,1)$ is relatively compact in $M$, $\inj{x} g \geq 1$ and 
   $\normes{\riem g(x)} \leq 1$. Then there exists a domain $(U)_1 \subset \tU \subset U$ such that for any $k > C(n)$ one can produce a Riemannian metric $h_k$ on $\tU$ with
   \begin{align*}
    & (\tU,h_k) \text{ is a complete Riemannian manifold, } \\
    & h_k \equiv g \text{ on }  (\tU)_{\frac{C(n)}{\sqrt{k}}}, \\
    & \normes{ \riem h_k(x) } \leq k \ \text{ and } \ \inj{x} h_k \geq \frac{1}{\sqrt{k}} \  \text{ for } x \in \tU.
   \end{align*}
  \end{lemma}
  
  \begin{remark}
   We have used the notation $(U)_r = \{ x \in U \ | \ B(x,r) \text{ is relatively compact in } U \}$.
  \end{remark}

 \begin{proof}
 
  Let us consider the real valued function $\rho$ defined on $U$, obtained through $\ref{lemme_regularisation}$ by smoothing the function
  \[ x \rightarrow \max(0, 2 - 4 d(x, M \setminus U) ). \]
  This function satisfies in particular $\rho \geq 0$, as well as   
  \begin{align*}
   & \rho \equiv 0 \text{ on } U_1, \\
   & \rho(x) \geq 1 \text{ on } \bound U, \\
   & \normes{ \nabla \rho }, \normes{ \nabla^2 \rho } \leq C(n).
  \end{align*}
  This justifies the choice of $\tU = \{ x \in U | \rho(x) < 1 \}$.
  
  \bigskip
  
  Next we fix $\epsilon>0$, and we define a metric $h$ on $\tU$, under the form $h=e^{2 f \circ \rho } g$ where $f: [0,1[ \rightarrow \setR_+$ is a smooth function that has to be chosen adequately. Let us consider to this effect
  the one real variable function\footnote{which is of course, up to rescaling and translation, nothing but the hyperbolic metric's conformal factor.}
  \[ \begin{array}{cc} 
    f(x) = & \left\{ \begin{array}{ll} 0 & \text{ for } 0 \leq x \leq 1-\epsilon, \\ 
                                        - \ln(1-(\frac{x-1+\epsilon}{\epsilon})^2) & \text{ for } 1-\epsilon<x<1, 
                     \end{array} \right. 
     \end{array} \]
  which can be smoothed in a neighborhood of $0$ into a function which we still call $f$ and which satisfies
  $f(x) = 0$ for $x \leq 1-\epsilon-\epsilon^2$,  $f(x) = - \ln(1-(\frac{x-1+\epsilon}{\epsilon})^2)$ for $x \geq 1-\epsilon+\epsilon^2$, and $0<f'(x) \leq \frac{2 \epsilon}{\epsilon^2-(x-1+\epsilon)^2}$ as well as 
  $0 < f''(x) \leq  \frac{4 \epsilon^2}{(\epsilon^2-(x-1+\epsilon)^2)^2}$ for $1-\epsilon-\epsilon^2 < x < 1$. At this point we simply recall the classical formula for conformal transformations (see for example \cite{Besse_2007}, page 58)
    \[ \sect{h}(\sigma) = \left( \sect{g}(\sigma) - \trace \nabla^2 (f \circ \rho) |_{\sigma} + \normes{ d(f \circ \rho)|_\sigma }^2 - \normes{ d(f \circ \rho) }^2 \right) e^{-2 f \circ \rho}. \]
  and we compute  
  \[ \normes{f' \nabla \rho}^2 e^{-2 f} \leq \frac{\normes{\nabla \rho}^2}{\epsilon^2}   \]
  then  
  \[ \normes{ \nabla^2 f \circ \rho } e^{-2 f} \leq \frac{4 \normes{\nabla \rho}^2}{\epsilon^2} + 2 \epsilon \normes{\nabla^2 \rho} \]
  whence finally $\sect{h}(\sigma) \leq \frac{C(n)}{\epsilon^2}$. 
  
  \bigskip
  
  In order to control the injectivity radius of $h$, we proceed through a lower bound of the volume of balls. The hypothesis on the metric $g$ implies a lower bound $\vol{g} B_g(x,r) \geq c(n) r^n$ for every $x \in \tU$, $0 < r < 1$. 
  Let us fix $x \in \tU$, write $d=d_g(x, M \setminus \tU)$, and let $0<\alpha<d$ to be fixed later. In the region $V= \{ y \in \tU \ | \ d - \alpha < d_g(y,M \setminus \tU) < d + \alpha \}$, one has
  \[ e^{2 f_\epsilon(1-d-\alpha)} g \leq h \leq e^{2 f_\epsilon(1-d+\alpha)} g, \]
  in particular, for $0 < r \leq \alpha e^{f_\epsilon(1-d+\alpha)}$, one has on the one hand $B_g(x, e^{ -f_{\epsilon} (1-d+\alpha) } r) \subset V$, while on the other hand $B_g(x, e^{-f_\epsilon(1-d+\alpha)} r) \subset B_h(x, r)$.
  Therefore,
  \begin{align*}
    \vol{h} B_h(x,r) = &  \int_{B_h(x,r)} \dvol{h} \\
    & \geq \int_{B_g(x, e^{-f_\epsilon(1-d+\alpha)} r)}  e^{n f_\epsilon(1-d-\alpha)} \dvol{g}, \\
    & \geq c(n) e^{-n (f_\epsilon(1-d+\alpha)-f_\epsilon(1-d-\alpha))} r^n.
  \end{align*}
  However $f_\epsilon(1-d+\alpha)-f_\epsilon(1-d-\alpha) \leq 2 \alpha f_\epsilon'(1-d+\alpha)$. By continuity, there exists $0< \alpha < d$ such that $\alpha = \frac{\epsilon}{4} e^{- f_\epsilon(1-d+\alpha)}$. A direct computation
  involving the expression for $f_\epsilon$ yields that for such an $\alpha$, $2 \alpha f_\epsilon'(1-d+\alpha)<1$, and thus that
  \[ \vol{h} B_h(x,r) \geq c(n) r^n \]
  for all $0 < r \leq \frac{\epsilon}{4}.$ Therefore one also get a lower bound $\inj{x} h \geq c(n) \epsilon$.

  \bigskip
  
  The lemma is obtained by choosing $\epsilon=\min( \sqrt{\frac{C(n)}{k}}, \frac{1}{c(n) \sqrt{k}})$.
 \end{proof}

\subsection{Construction of the flow} \label{subsection_partial_flow}

 The existence of a ``partial flow'' for any Riemannian manifold $(M^n,g_0)$, as described in the statement of theorem $\ref{theorem_partial_flow}$, actually comes from the following more detailed statement.
 
 \begin{proposition} \label{proposition_piecewise_flow}
  There exist $C(n),c(n)>0$ with the following property. Let $(M^n,g_0)$ be a Riemannian manifold (not necessarily complete), $T>0$, $K_0 \geq 1$ and $K_1 \geq C K_0$. Then for the sequence
  \[ \dots < t_i < \dots < t_{-1} < t_0 = T \]
  defined by $t_i=\left(1+\frac{c}{K_1} \right)^i T$ (in particular $t_i \underset{i \rightarrow -\infty}{\longrightarrow} 0$) there exist
    
  \begin{itemize}

   \item[(i)] An exhaustion of $M$ by open subsets, in other words a sequence
   \[ \dots \supset U_i \supset \dots \supset U_{-1} \supset U_0, \]
   with $\bigcup_{i=-\infty}^0 U_i =M$ (although one allows $U_i = \emptyset$ from some rank on),
   
   \bigskip
   
   \item[(ii)] A smooth solution $g$ to the Ricci flow equation defined on
   \[ \calD = M \times \{ 0 \} \cup \bigcup_{i=-\infty}^0 U_i \times [t_{i-1},t_i], \]   
   with initial condition
   \[ \begin{array}{cc} g(x,0)=g_0(x), & \text{ for every } x \in M, \end{array} \]
   satisfying the estimates
   \[ \begin{array}{cc} 
    \normes{ \riem g(x,t) } \leq \frac{K_1}{t}, & \inj{x} g(t) \geq r,
   \end{array} \]
   for all $(x,t) \in \calD$ such that $0 < r \leq \sqrt{ \frac{t}{K_1} }$ and $B_{g(t)}(x,r) \times \{ t \}$ is relatively compact in $\calD$.   
  \end{itemize}
  
  \bigskip
  
  Moreover, the construction possesses the following ``maximality'' property
  
  \bigskip
  
  \begin{itemize}
   \item[(iii)] For every $i < 0$, if $x \in U_i$ is such that $B_{g_0}(x, C \sqrt{K_1 t_i})$ is relatively compact in $U_i$, and for any $y \in B_{g_0}(x, C \sqrt{K_1 t_i})$ one has
   \[ \begin{array}{cc} 
    \normes{ \riem g(y,t_i) } < \frac{K_0}{t_i}, & \inj{g(t_i)} y > \sqrt{ \frac{t_i}{K_0} },
   \end{array} \]
   then $x \in U_{i+1}$.
  \end{itemize}

 \end{proposition} 
 
 As a straightforward consequence of the above, we get the
 
 \begin{proposition} \label{proposition_partial_flow}
  There exists constants $c(n),C(n)$ with the property that, for any choice of $K_0 \geq 1$, $K_1 \geq C(n)K_0$ and any Riemannian manifold $(M,g_0)$, there exists a partial flow of initial data $g_0$ and parameters $K_0$, $K_1$.
 \end{proposition}
 
  \begin{remark}
   Let us stress again that the proposition above gives no control on the evolution of the domain $ \calD_t = \calD \cap M \times \{ t \}$ on which the metric $g(t)$ is defined at time $t$. We know that $\calD$ is open, non-empty 
   and contains the initial time slice $M \times \{ 0\}$, but in general, $\calD_t$ can ``evaporate'' arbitrarily fast. Correspondingly, by proposition $\ref{proposition_piecewise_flow}$, we only know that the $U_i$ eventually 
   become non-empty when $i \rightarrow - \infty$ - they can be empty from an arbitrarily negative index $i$ (corresponding to arbitrarily small time $t_i$).
 \end{remark} 

 The flow described in proposition $\ref{proposition_piecewise_flow}$ is in turn obtained as the limit of the following finite constructions
 
 \begin{proposition} \label{proposition_finite_piecewise_flow}
  There exists $C(n), c(n) >0$ with the following property. Let $(M^n,g_0)$ be a Riemannian manifold (not necessarily complete), $T>0$, $K_0 \geq 1$, $K_1 \geq C K_0$ and $\calK \subset M$ compact. Then there exists $k \leq 0$ 
  such that, for the finite sequence of times
  \[ t_{k-1} < \dots < t_i < \dots < t_{-1} < t_0 = T \]
  defined by $t_i=\left(1+\frac{c}{K_1} \right)^i T$ for $k-1 \leq i \leq 0$, there exist
  
  \bigskip
  
  \begin{itemize}
   \item[(i)] A finite sequence 
   \[ V_k \supset \dots \supset V_i \supset \dots \supset V_{-1} \supset V_0, \]
   as well as an inner sequence
   \[ U_k \supset \dots \supset U_i \supset \dots \supset U_{-1} \supset U_0, \]
   of open sets of $M$ with $U_i \subset V_i$ for $k \leq i \leq 0$, and $U_k \supset \calK$,
   
   \bigskip
   
   \item[(ii)] For any $k \leq i \leq 0$ a complete Ricci flow $g_i$ defined on $V_i \times [t_{i-1},t_i]$ satisfying the estimates 
   \[ \begin{array}{cc} 
    \normes{ \riem g_i(x,t) } \leq \frac{K_1}{t}, & \inj{x} g_i(t) \geq \sqrt{\frac{t}{K_1}},
   \end{array} \]
   for all $(x,t) \in V_i \times [t_{i-1},t_i]$, as well as the gluing conditions
   \[ g_{i-1}(x,t_{i-1}) = g_{i}(x,t_{i-1}) \]
   for $k+1 \leq i \leq 0$, $x \in U_i$, and the initial condition $g_k(x,t_{k-1})=g_0(x)$ on $U_k$.
  \end{itemize}
  
  \bigskip
  
  Here also, the construction possesses a ``maximality'' property
  
  \bigskip
  
  \begin{itemize}
   \item[(iii)] For every $k \leq i < 0$, if $x \in U_i$ is such that $B_{g_0}(x, C \sqrt{ K_1 t_i })$ is relatively compact in $U_i$ and if for any $y \in B_{g_0}(x, C \sqrt{ K_1 t_i })$ one has
   \[ \begin{array}{cc} 
    \normes{ \riem g_i(y,t_i) } < \frac{K_0}{t_i}, & \inj{g_i(t_i)} y > \sqrt{ \frac{t_i}{K_0} },
   \end{array} \]
   then $x \in U_{i+1}$.
  \end{itemize}

 \end{proposition}

 The proof of proposition $\ref{proposition_finite_piecewise_flow}$ essentially consists in the repeated application of Shi's theorem $\ref{theorem_shi}$ combined with appropriate modification of the metric at times $t_i$ as 
 allowed by lemma $\ref{lemme_chirurgie}$. Recall that
 
 \begin{theorem}[Shi's existence theorem \cite{Shi_1989}] \label{theorem_shi}
  There exists a constant $c(n)>0$ with the following property. Let $(M^n,g_0)$ be a complete Riemannian manifold with $\normes{ \riem g_0 } \leq 1$ on $M$. Then there exists a complete Ricci flow $g(t)$ $M \times [0, c^2]$, with
  $\normes{ \riem g(x,t) } \leq 2$ for any $(x,t) \in M \times [0,c^2]$. Additionally, $c(n)$ can be chosen such that if $\inj{g_0}(x) \geq 1$ for every $x \in M$, then $\inj{g(t)}(x) \geq \frac{1}{4}$ for any $(x,t) \in M \times 
  [0,c^2]$.
 \end{theorem}
 
 \begin{proof}[Proof of proposition $\ref{proposition_finite_piecewise_flow}$]
  Define the sequence $T=t_0> ... > t_k$ as in the statement of the proposition, where $c=c(n)$ is given by theorem $\ref{theorem_shi}$, and $k<0$ is to be chosen below.
  
  \bigskip
  
  Suppose the sets $U_i, V_i$ and the flow $g_i(t)$ have been constructed up to step $i$. Consider the open domain of $M$ defined by
  \[ \tV_i = \left\{ x \in  \ U_{i-1} \  \left| \  \normes{\riem{g_{i-1}(x,t_{i-1})}} < \frac{K_0}{t_{i-1}}, \ \inj{x} g_{i-1}(t_{i-1}) >  \sqrt{ \frac{t_{i-1}}{K_0} } \ \right. \right\}. \]
  In the base case $i=k$ one replaces $U_{i-1}$ by $M$ and $g_{i-1}(t_{i-1})$ by $g_0$ in the definition of $\tV_k$. Clearly, if $k$ is chosen negative enough, one has $\calK \subset \tV_k$ by continuity. On the other hand, for 
  every $i>k$, $\tV_i$ can be empty, and the construction empty from this step on.
   
   \bigskip
   
   Since $\tV_i \subset U_{i-1}$, the metric $g_{i-1}(t_{i-1})$ is defined on $\tV_i$ and satisfies the estimates 
   \begin{equation}\label{prop_fpw_0}
    \begin{array}{cc} \normes{ \riem g_{i-1}(t_{i-1}) } < \frac{K_0}{t_{i-1}}, & \inj{x} g_{i-1}(t_{i-1}) >  \sqrt{ \frac{t_{i-1}}{K_0} } \end{array}
   \end{equation}
   for any $x \in \tV_i$ by the choice of this domain. Assume $K_1 \geq 16 C_0 K_0$, where $C_0=C(n)$ is given by lemma $\ref{lemme_chirurgie}$. Applying the lemma with $k=\frac{K_1}{16 K_0}$ to the metric $g_{i-1}(t_{i-1})$ 
   (actually, to the scaled up metric $\sqrt{ \frac{K_0}{t_{i-1}} } g_{i-1}(t_{i-1})$) provides us with a modified metric $g_{i,0}$ complete on a sub-domain $V_i \subset \tV_i$ satisfying the following properties:
   \begin{equation} \label{prop_fpw_1}
    V_i \supset \left\{ x \in \tV_i \ \left| \ d_{g_{i-1}(t_{i-1})}(x, M \setminus \tV_i) > \sqrt{ \frac{t_{i-1}}{K_0} }  \right. \right\},
   \end{equation}
   and if one defines a sub-domain of $V_i$ by
   \begin{equation} \label{prop_fpw_2}
    U_i = \left\{ x \in V_i \ \left| \ d_{g_{i-1}(t_{i-1})}(x, M \setminus V_i) > 4 \sqrt{ \frac{C_0 t_{i-1}}{K_1} }  \ \right. \right\},
   \end{equation}
   then on $U_i$ one has
   \begin{equation} \label{prop_fpw_3}
    g_{i,0} \equiv g_{i-1}(t_{i-1}).
   \end{equation}
   Moreover, the modified metric satisfies the estimates 
   \[ \begin{array}{cc} \normes{ \riem g_{i,0} } < \frac{K_1}{16 t_{i-1}}, & \inj{x} g_{i,0} >  4 \sqrt{ \frac{t_{i-1}}{K_1} }, \end{array} \]
   for every $x \in V_i$.

  \begin{center}
    \includegraphics[scale=1]{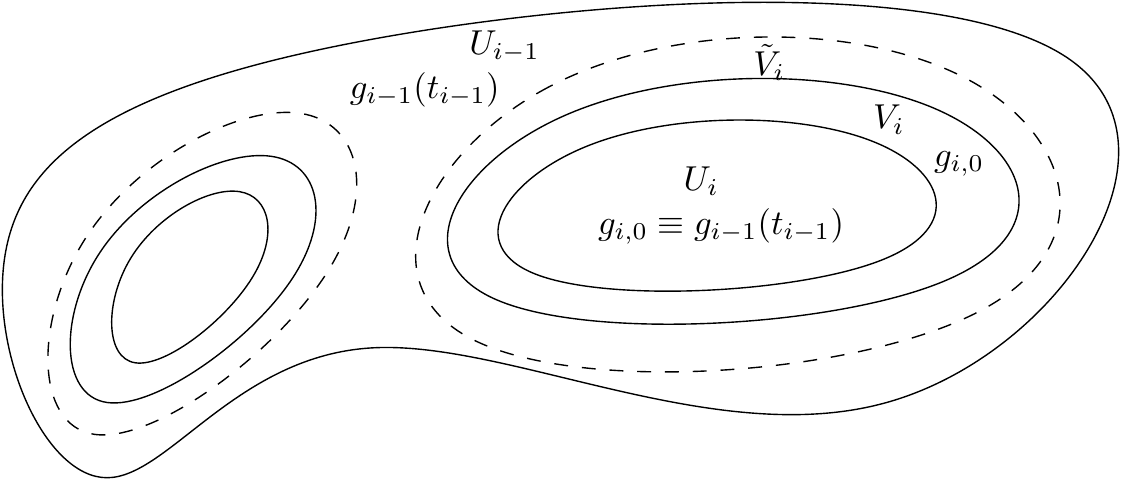}
  \end{center}
   
   Shi's existence theorem then asserts the existence a complete Ricci flow $g_i(t)$ on $V_i \times [t_{i-1},t_i]$ with initial data $g_{i,0}$, satisfying the estimates    
   \begin{equation} \label{prop_fpw_4}
    \begin{array}{cc} \normes{ \riem g_i(x,t) } < \frac{K_1}{t}, & \inj{x} g_i(x,t) >  \sqrt{ \frac{t}{K_1} }, \end{array} 
   \end{equation}   for $(x,t) \in V_i \times [t_{i-1},t_i]$ (we used $\frac{t_i}{t_{i-1}} \leq 4$).
   
   \bigskip
   
   $(\ref{prop_fpw_3})$ and $(\ref{prop_fpw_4})$ correspond to $(ii)$ in the statement. To verify part $(iii)$, we need first remark that $(\ref{prop_fpw_1})$ and $(\ref{prop_fpw_2})$ imply in particular
   \begin{equation}\label{prop_fpw_5}
    U_i \supset \left\{ x \in \tV_i \ \left| \ d_{g_{i-1}(t_{i-1})}(x, M \setminus \tV_i) >  2 \sqrt{\frac{t_{i-1}}{K_0}} \ \right. \right\},
   \end{equation}
   and then note that we have the following distances comparison
   \[ d_{g_{i-1}(t_{i-1})} (x, M \setminus \tV_i) \geq d_{g_0}(x, M \setminus \tV_i) - \frac{20}{3}(n-1) \sqrt{ K_1 t_{i-1} }. \]
   simply because, thanks to the gluing property $(\ref{prop_fpw_3})$, defining $g(x,t)$ on $U_{i-1} \times [t_{k-1},t_{i-1}]$ by $g(x,t) = g_{j}(x,t)$ if $t_{j-1} \leq t \leq t_j$ yields a smooth non-complete flow, to 
   which we can apply lemma $\ref{lemme_distorsion_distances}$. Thus $(\ref{prop_fpw_5})$ implies, in term of the metric $g_0$,
   \begin{equation}\label{prop_fpw_6}
    U_i \supset \{ \ x \in \tV_i \ | \ d_{g_0}(x, M \setminus \tV_i) >  C_1 \sqrt{ K_1 t_{i-1} } \ \}.
   \end{equation}
   for $C_1=2 +\frac{20}{3}(n-1)$ (so that $\frac{2}{\sqrt{K_0}} + \frac{20}{3} (n-1) \sqrt{K_1} \leq C_1 \sqrt{K_1}$). Now for every point $x \in M$ such that 
   $B_{g_0}(x, C_1 \sqrt{K_1 t_{i-1}}) \subset U_{i-1}$ and such that the estimates $(\ref{prop_fpw_0})$ hold on this ball, one has $B_{g_0}(x, C_1 \sqrt{K_1 t_{i-1}}) \subset \tV_i$ by definition of $\tV_i$. Thus by 
   $(\ref{prop_fpw_6})$,  $x \in U_i$. This is property $(iii)$ at rank $i-1$ where the constant of the statement is $C=\max(C_0,C_1)$.

 \end{proof} 
 
 \begin{proof}[Proof of proposition $\ref{proposition_piecewise_flow}$]
  For each value of $k \leq 0$ we consider the finite construction described in the statement of proposition $\ref{proposition_finite_piecewise_flow}$. This provides us in particular, for each value of $k \leq 0$, with two 
  decreasing sequences $U_i^k$, $V_i^k$, $k \leq i \leq 0$ of open sets with $U_i^k \subset V_i^k$, as well a sequence of flows $g_i^k$ on $V_i^k \times [t_{i-1},t_i]$ for $k \leq i \leq 0$.
 \begin{itemize}
 
  \item[\bf (i)] Let us start with some definitions. For each $k \leq 0$ one considers the space-time domain $D^k \subset M \times [0,T]$ defined by
   \[ D^k = \bigcup_{i=k}^0 U_i^k \times [t_{i-1},t_i[ \]
   on which one can define a smooth solution to the Ricci equation $g^k(x,t)$ (with non complete time slices), by setting   
   \[ g^k(x,t) = g_i^k(x,t) \]
   for $t \in [t_{i-1}, t_i]$, $x \in U_i^k$. By construction, the flow $g^k$ satisfies in particular the bounds
   \begin{equation} \label{borne_courbure_gk}
    \begin{array}{cc} \normes{\riem g^k(x,t)} \leq \frac{K_1}{t}, & \inj{g^k(t)} x \geq r , \end{array}
   \end{equation}
   for every $(x,t) \in D^k$ and $0 < r \leq \sqrt{ \frac{t}{K_1} }$ such that $B_{g^k(t)}(x,r) \times \{ t \}$ is relatively compact in $D^k$. 
   
   \bigskip
   
   Furthermore we define the corresponding ``limit'' domains:
   \begin{equation}\begin{split} \label{def_U_infty}
    U_i^\infty = \{ x \in M \ | \ \exists k_0 \leq 0  \text{ such that } & \forall k \leq k_0, \\ & B_{g_0}(x,C_1 \sqrt{K_1 t_i}) \text{ is relatively compact in } U_i^k \},
   \end{split}\end{equation}
   where $C_1=C_1(n)$ will be chosen in the course of the argument (see step {\bf (iii)}), 
   \[ D^\infty = \bigcup_{i=-\infty}^0 U_i^\infty \times [t_{i-1},t_i[, \]
   as well as
   \[ \calD^\infty = D^\infty \cup M \times \{ 0 \}. \]
   Clearly $U_i \subset U_{i-1}$, although at this stage all the $U_i$ and thus $D$ could be empty.
   
   \item[\bf (ii)] The first thing we check is that the $U_i$ do exhaust $M$ and thus that $\calD$ is a non-empty open domain of $M \times [0,T]$. For $\calK \subset M$ compact we pick $r_\calK>0$ a scale such that   
   \[ \begin{array}{cc} \normes{ \riem g_0(x) } \leq \frac{1}{r_\calK^2}, & \inj{g_0} x \geq r_\calK, \end{array} \]   
   for $x \in \calK$, and we show that there exists $A(n,K_1)$ and $\epsilon_\calK=\epsilon(r_\calK,n,K_1)$ such that   
   \begin{equation} \label{contraction_du_domaine}
    U_i^\infty \supset (\calK)_{A \sqrt{t_i}} = \{ x \in \calK, \ | \  B_{g_0}(x, A \sqrt{t_i}) \subset \calK \}, 
   \end{equation}
   as soon as $i$ is such that $t_i \leq \epsilon_\calK^2$. Clearly this implies that $\bigcup_{i=-\infty}^0 U_i^\infty = M$.
   
   \bigskip
   
   In order to prove $(\ref{contraction_du_domaine})$, we use the maximality property $(iii)$ from proposition $\ref{proposition_finite_piecewise_flow}$, together with corollary $(\ref{misc_1})$. Therefore we let $\epsilon = 
   \min(\epsilon_{\ref{misc_1}}(n,K_1), \frac{1}{2})$, and we fix $k \leq i \leq 0$. Then, for any point $x \in M$ such that $B_t(x,  \frac{\sqrt{t_i}}{\epsilon} ) \subset U_i^k \cap \calK$ for $t_{k-1} \leq t \leq t_i$, and if
   $t_i \leq r_\calK^2 \epsilon^2$, then the hypothesis of corollary $\ref{misc_1}$ are satisfied by the rescaled flow $\tg^k(s)= \frac{\epsilon^2}{t_i} g^k( \frac{t_i}{\epsilon^2} s +t_{k-1})$ for $s \in 
   [0, \epsilon^2 (1- \frac{t_{k-1}}{t_i})]$. Whence in particular the estimates at $(x,t_i)$ 
   \begin{equation} \label{non_doubling_bounds}
    \begin{array}{cc} \normes{ \riem g^k(x,t_i) } \leq \frac{K_0}{4 t_i}, & \inj{g^k(t_i)} x \geq 2 \sqrt{ \frac{t_i}{K_0} }, \end{array}
   \end{equation}
   (since $\frac{\epsilon^2}{t_i} \leq \frac{K_0}{4 t_i}$). However distance comparison lemma $\ref{lemme_distorsion_distances}$ apply to the flow $g^k(t)$ on $U_i^k \times [t_{k-1},t_i]$, and
   \[ d_t(x, M \setminus U_i) \geq d_0(x, M \setminus U_i) - \frac{20}{3} (n-1) \sqrt{K_1 t}. \]
   Thus if $x \in U_i$ is such that $B_{g_0}(x, (\frac{1}{\epsilon} + \frac{20}{3} (n-1) \sqrt{K_1} ) \sqrt{t_i}) \subset U_i^k \cap \calK $, $(\ref{non_doubling_bounds})$ hold at $x$, and if finally $x$ is such that $B_{g_0}(x, 
   a \sqrt{t_i}) \subset U_i^k \cap \calK $, where $a=a(n,K_1)=\frac{1}{\epsilon} + \frac{20}{3} (n-1) \sqrt{K_1} + C_0 \sqrt{K_1}$ and $C_0=C(n)$ comes from proposition $\ref{proposition_finite_piecewise_flow}$ then 
   $(\ref{non_doubling_bounds})$ hold on $B_{g_0}(x, C_0 \sqrt{K_1 t_i})$. Thus by property $(iii)$ of proposition $\ref{proposition_finite_piecewise_flow}$,
   \[ U_{i+1}^k \cap \calK \supset \left( U_i^k \cap \calK \right)_{a \sqrt{t_i}} \]   
   holds true as long as $t_i \leq \epsilon_\calK^2=\epsilon^2 r_\calK^2$. This in turn integrates into
   \[ U_i^k \cap \calK \supset (U_k^k \cap \calK)_{A' \sqrt{t_i}} \]
   for some $A'=A'(n,K_1)$ independent from $\calK$. Recalling that proposition $\ref{proposition_finite_piecewise_flow}$ also guaranties $U_k^k \supset \calK$ for $k$ negative enough, letting $k \rightarrow -\infty$ we get 
   $U_i^\infty \supset (\calK)_{(A' + C_1 \sqrt{K_1}) \sqrt{t_i}}$ for $i$ such that $t_i \leq \epsilon_\calK^2$, which is $(\ref{contraction_du_domaine})$ with $A=A'+C_1 \sqrt{K_1}$.
   
   \bigskip
   
   \item[\bf (iii)] We are now ready to extract a subsequence of the $g^k(t)$ which converges in $\calC^m$-norm, for any $m \geq 0$ and on every compact $\calK \times [\tau',\tau] \subset \calD^\infty$, to a smooth limit $g^\infty(t)$ defined on 
   $\calD^\infty$. So let $\calK \times [0,\tau] \subset \calD^\infty$ be fixed. The definition of the $U_i^\infty$ and a straightforward covering argument ensure that for some $k$ negative enough, one has, for every 
   $x \in \calK$, $B_{g_0}(x, C_1 \sqrt{K_1 \tau}) \subset U_i^k$ where $i \leq 0$ is such that $t_{i-1} \leq \tau < t_i$. So let us pick $x \in \calK$ and note that, by distance comparison corollary $\ref{lemme_distorsion_distances}$,
   we have, for every $t_{k-1} \leq t \leq \tau$,
   \[ d_{g^k(t)}(x, M \setminus U_i^k) \geq d_0(x, M \setminus U_i^k) - \frac{20}{3} (n-1) \sqrt{K_1 t} \]
   Thus the choice of $C_1 = 1 + \frac{20}{3}(n-1)$ in $(\ref{def_U_infty})$ guaranties in particular that $B_{g^k(t)}(x,\sqrt{K_1 \tau})$ is relatively compact in $D^k$ for $t_{k-1} \leq t \leq \tau$. Therefore we can apply 
   proposition $\ref{misc_0}$ from appendix B to deduce the existence of constant $\Lambda_{\calK,\tau}$, depending on $g_0$, $K_1$, but independent from $k$ such that
   \begin{equation} \label{bound_curvature}
    \normes{ \riem g^k(x,t) } \leq \Lambda_{\calK,\tau},
   \end{equation}
   for $x \in \calK$, $t_{k-1} \leq t \leq \tau$. This implies first, simply by integrating the Ricci flow equation, the existence of a constant $\Lambda_{0,\calK,\tau}$ such that
   \begin{equation} \label{bound_metric}
    (\Lambda_{0,\calK,\tau})^{-1} g_0(x) \leq g^k(x,t) \leq \Lambda_{0,\calK,\tau} g_0(x)
   \end{equation}
   on the same $\calK \times [t_{k-1}, \tau]$. Secondly, making use of the local version of Shi's estimates (as stated in \cite{Chow_2008}, thm 14.16) one finds, for every $m \geq 1$, a constant $\Lambda_{m,\calK,\tau}$ such that
   \begin{equation} \label{bound_derivatives}
    \normes{ \nabla^m g^k(x,t) }_{g_0} \leq \Lambda_{m,\calK,\tau}
   \end{equation}
   for $x \in K$, $t_{k-1} \leq t \leq \tau$. With this in hand, Arzela-Ascoli compactness theorem allows us to extract a limit flow $g^\infty$ defined on $D^\infty$ such that the $g^k$ converge toward $g^\infty$ in the $\calC^m$-norm
   on any compact subset of $D^\infty$.
   
   \bigskip
   
   Finally, $g^\infty$ can be extended to $\calD^\infty$ by setting $g^\infty(x,0)=g_0(x)$ for $x \in M$. To ensure that this defines an initial condition in the usual sense one needs then check that on 
   every compact $\calK \subset M$ $g^\infty(x,t)$ converges to $g_0$ in the $\calC^m$-norm for any $m \geq 0$ when $t$ tends to $0$. But this, again, is a consequence of the uniform bounds on the curvature
   tensor and its derivatives established above. Indeed, using $(\ref{bound_curvature})$, $(\ref{bound_metric})$ and integrating between $t_{k-1}$ and $t$ it is straightforward to prove that
   \[ \normes{ g^k(t)-g_0 }_{g_0} = \normes{ g^k(t)-g^k(t_{k-1}) }_{g_0} \leq \Lambda'_{\calK,\tau} t \]
   for some constant $\Lambda'_{\calK,\tau}$ independent from $k$, while $(\ref{bound_derivatives})$ together with a bit more work yields that
   \[ \normes{ \nabla^m g^k(t) - \nabla^m g_0 }_{g_0} \leq \Lambda'_{m,\calK,\tau} t. \]
   
   \item[\bf (iv)] At this stage it is clear that the construction features all the properties announced in the statement. The estimates in $(ii)$ are a consequence of $(\ref{borne_courbure_gk})$ and of 
   the smooth convergence of the $g^k$ towards the limit flow $g^\infty$. $(iii)$ is also inherited from the corresponding property in the finite construction. Consider indeed a point $x \in M$ such that
   for some rank $i\leq 0$, $B_{g_0}(x,C_2 \sqrt{K_1 t_i}) \subset U_i^\infty$, where $C_2>2 C_1+C$, $C=C(n)$ being the constant from proposition $\ref{proposition_finite_piecewise_flow}$, and such that 
   $\normes{ \riem g^\infty(x,t) } < \frac{K_0}{t_i}$, $\inj{g^\infty(t_i)} x > \sqrt{ \frac{t_i}{K_0} }$. Then $B_{g_0}(x, (2 C_1+C) \sqrt{K_1 t_i}) \subset U_i^k$ for $k$ negative enough, and by the smooth
   convergence of the $g^k$ towards $g^\infty$, $\normes{ \riem g^k(y,t_i) } < \frac{K_0}{t_i}$, $\inj{g^k(t_i)} y > \sqrt{ \frac{t_i}{K_0} }$ for every $y \in B_{g_0}(x, (2 C_1+C) \sqrt{K_1 t_i})$. By 
   property $(iii)$ of the finite construction, this implies $B_{g_0}(x, 2 C_1 \sqrt{K_1 t_i} ) \subset U_{i+1}^k$, whence $B_{g_0}(x, C_1 \sqrt{K_1 t_{i+1}}) \subset U_{i+1}^k$ (we used 
   $\frac{t_{i+1}}{t_i} \leq 2$). This in turn, implies by definition that $x \in U_{i+1}^\infty$ (thus, the $C(n)$ of the statement is obtained as $\max(C_1,C_2)$).

 \end{itemize}

 \end{proof}
 
 \begin{proof}[Proof of proposition $\ref{proposition_partial_flow}$]
  Let $K_0 \geq 1$ and $(M,g_0)$ be fixed. We consider the flow $g$ on a space time region $\calD \subset M \times [0,1]$, whose existence is guarantied by proposition $\ref{proposition_piecewise_flow}$ applied with parameters
  $4 K_0$, $K_1=4 C_0 K_0$ ($C_0=C(n)$ comes from the proposition) and $T=1$, and we show that it has indeed the properties of a partial flow of parameter $K_0$ required by the statement of $\ref{theorem_partial_flow}$. Property 
  {\it (ii)} is obvious for $C=2 C_0$, only {\it (iii)} remains to be checked.
  
  \bigskip
  
  Consider thus an open domain $U \subset M$ and $\tau>0$ such that $U \times [0,t] \subset \calD$ and such that for any $(x,t) \in U \times [0,\tau]$, 
  \begin{equation} \label{the_estimates}
   \begin{array}{cc} \normes{ \riem g(x,t) } \leq \frac{K_0}{t}, & \inj{g(t)} x \geq \sqrt{\frac{t}{K_0}}. \end{array}
  \end{equation}
  Let $i \leq 0$ be such that $t_{i-1} < \tau \leq t_i$. $(\ref{the_estimates})$ at time $t_{i-1}$ imply by property $(iii)$ of proposition $\ref{proposition_piecewise_flow}$,
  \[ (U)_{C_0 \sqrt{K_1 t_{i-1}}} \subset U_i. \]
  Now we apply proposition $\ref{misc_1}$ from Appendix B to get similar estimates at times $t_i$. Consider the rescaled flow defined by
  \[ h(x,s) = \frac{K_0}{\tau} g(x, \frac{\tau}{K_0} s + \tau) \]
  for $(x,s) \in U_i \times [0, K_0 \frac{t_i-\tau}{\tau}]$. In term of $h$, $(\ref{the_estimates})$ at time $\tau$ yields
  \[ \begin{array}{cc} \normes{ \riem h(x,0) } \leq 1, & \inj{h(0)} x \geq 1,    \end{array} \]
  for $x \in U \cap U_i$, while the $\sqrt{\frac{t}{K_1}}$-regularity scale control which holds for $g$ by assumption translates after rescaling into a $\sqrt{\frac{ s + K_0}{K_1}}$-regularity scale
  estimate at time $s$ for the metric $h(s)$. Noting that $s \leq K_0 \frac{t_i-t_{i-1}}{t_{i-1}} \leq \frac{c K_0}{K_1}$, and recalling that we chose $K_1 = 4 C_0 K_0$, this yields
  \[ \begin{array}{cc} \normes{ \riem h(x,s) } \leq \frac{4 C_0}{s}, & \inj{h(s)} x \geq \sqrt{\frac{s}{4 C_0}}. \end{array} \]
  Moreover, if one assumes that $x \in M$ is such that $B_{g_0}(x,C_1 \sqrt{K_1 \tau}) \subset U \cap U_i$ with $C_1=\frac{20}{3}(n-1)+1$ then by distance comparison one guaranties in particular that for 
  $0 \leq s \leq K_0 \frac{t_i-\tau}{\tau}$,  $B_{h(s)}(x,\frac{\tau}{K_0})$ is relatively compact $U \cap U_i$. Thus proposition $\ref{misc_1}$ provides us with $\epsilon(n)=\epsilon(n,C_0)$ such that
  \[ \begin{array}{cc} \normes{ \riem h(x,s) } \leq 4, & \inj{h(s)} x \geq \frac{1}{2},    \end{array} \]
  on $\left( U \cap U_i \right)_{C_1 \sqrt{K_1 \tau}}$ as long as $s \leq \epsilon(n)$. Since $0 \leq s \leq K_0 \frac{t_i-\tau}{\tau} \leq \frac{c}{4 C_0}$ and since one can always assume 
  $\frac{c}{4 C_0} \leq \epsilon(n)$ one deduces in term of th flow $g$ at time $t_i$
  \[ \begin{array}{cc} \normes{ \riem g(x,t_i) } < \frac{4K_0}{t_i}, & \inj{h(s)} x > \sqrt{\frac{t_i}{2 K_0}},    \end{array} \]
  on the same domain. This implies $\left( U \cap U_i \right)_{(C_1+C_0) \sqrt{K_1 t_i}} \subset U_{i+1}$ by maximality. Finally
  \[ ( U )_{(C_1+2C_0) \sqrt{K_1 t_i}} \subset U_{i+1} \]
  Thus $\calD$ contains in particular the space-time domain $(U)_{\Delta \rho} \times [\tau, \tau+\Delta \tau]$ with  $\Delta \rho = (C_1+2 C_0) \sqrt{K_1 \tau}$ et $\Delta \tau = \frac{c}{K_1}{\tau}$ (the $C(n)$ of the statement
  being finally chosen as $C_1+2 C_0$).
 \end{proof}

 \section{Proof of theorem $\ref{main_theorem}$}
 
  We are now ready to prove $\ref{main_theorem}$.
 
 \begin{proof}[Proof of theorem $\ref{main_theorem}$.]
  Recall that we consider $(M^3,g_0)$ a Riemannian manifold which need not be complete, with $\ricci g_0 \geq -1$ on $M$ and such that for all $x \in M$, $r \leq 1$ such that $B_{g_0}(x,r)$ is relatively compact in $M$, one has
  $\vol{g_0} B(x,r) \geq v_0 r^3$. The argument for proving theorem $\ref{main_theorem}$, sketched in the outline, consists in using the enhanced regularity scale estimates provided by proposition $\ref{main_proposition}$ together
  with the so-called maximality property of the partial flow to show that a partial flow of $M$ actually contains a domain of the form
  \[ \calM_{A,\epsilon^2} = \{ (x,t) \in M \times[0,\epsilon^2] \ | \ B_t(x,A\sqrt{t}) \text{ is relatively compact in } M \} \]
  for an appropriate choice of the parameter $K_0$ in theorem $\ref{theorem_partial_flow}$.
  
  \bigskip
  
  We put\footnote{recall that $\iota(n,v)$ is the constant appearing in Cheeger-Gromov-Taylor's lower bound on the injectivity radius. See note page \pageref{Cheeger_Gromov_Taylor}.} $K_0= \frac{4 K}{\iota(3,v)^2}$, 
  where $K=K(v_0)$ and $v=v(v_0)$ come from the statement of proposition $\ref{main_proposition}$, and we let $g$ be a partial flow on a domain $\calD \subset M \times [0,1]$, the existence of which is guarantied by theorem 
  $\ref{theorem_partial_flow}$ for this choice of the parameter. Then, for fixed $A>0$ (the value of which will be chosen later) we let $\tau$ be the supremum of those times such that  
  both $\calM_{A,\tau} \subset \calD$ and the estimates
  \begin{equation} \label{final_proof_estimates}
   \begin{array}{ccc} \normes{ \riem g(x,t) } \leq \frac{K}{t}, & \vol{} B_t(x,r) \geq v r^3, & \ricci g(x,t) \geq - \frac{1}{t}, \end{array}
  \end{equation}
  hold for every $(x,t) \in \calM_{A,\tau}$. We set $\epsilon=\epsilon_0(v_0, C^2 K_0)$ where $\epsilon_0$ comes from the statement of proposition $\ref{main_proposition}$, and $C=C(3)$ comes from theorem $\ref{theorem_partial_flow}$.
  We make the assumption $\tau < \epsilon^2$ and show that if the constant $A$ is chosen adequately (only depending on $v_0$) we get a contradiction, which, clearly, establishes the statement.
  
  \bigskip
  
  To this effect, let us pick $\tau^- < \tau < \tau^+$ such that $\tau = \tau^-+ \frac{1}{2} \Delta \tau^-$ and $\tau^+=\min(\epsilon^2,\tau^-+\Delta \tau^-)$, i.e. $\tau^- = (1+\frac{1}{2 C^2 K_1})^{-1} \tau$, $\tau^+ = 
  \min(\epsilon^2, (1+\frac{1}{C^2 K_1}) (1+\frac{1}{2 C^2 K_1})^{-1} \tau)$. The hypothesis on $g_0$ as well as the $C^{-1} \sqrt{ \frac{t}{K_0} }$-regularity scale control which holds by construction for the partial flow allow us 
  to apply proposition $\ref{main_proposition}$ with $K_1 = C^2 K_0$ to the domain $(M)_{A \sqrt{\tau^-}} \times [0,\tau^-] \subset \calD$. As a result, and since we assumed $\tau \leq \epsilon^2$ we get that the estimates $(\ref{final_proof_estimates})$
  hold for every $0 < t \leq \tau^-$, $x \in (M)_{(A+a) \sqrt{\tau^-}}$ ($a$ is also given by proposition $\ref{main_proposition}$) and $0< r \leq \sqrt{t}$. By the choice of $K_0$ and Cheeger-Gromov-Taylor, 
  this implies in particular a $\sqrt{ \frac{t}{K_0} }$-regularity scale control for $g(t)$ on this domain for $0 < t \leq \tau^-$. The maximality property of the partial flow then guaranties the inclusion $(M)_{(A+a') 
  \sqrt{\tau^-}} \subset \calD_{\tau^+}$ for $a'=a+C \sqrt{K_0}$.
  
  \begin{center}
   \includegraphics{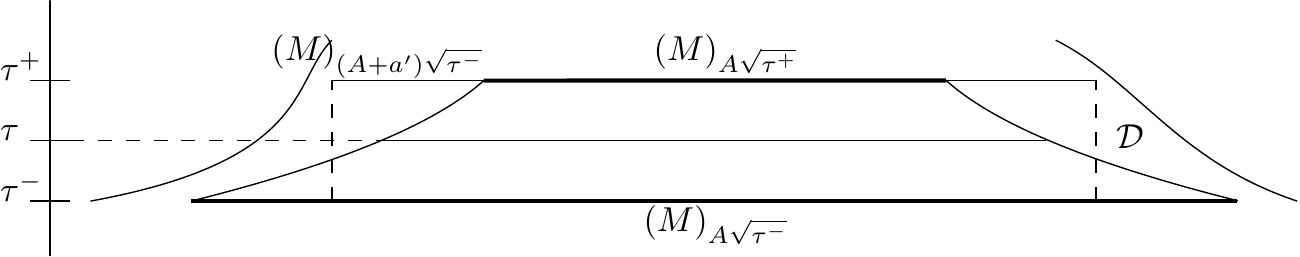}
  \end{center}
  
  Now if $A$ is such that $(a+a'+A) \sqrt{\tau^-} \leq A \sqrt{\tau}$ (which justifies the choice we finally make of $A=A(v_0)= (a+a') \left(\sqrt{1+\frac{1}{2 C^2 K_0}}-1\right)^{-1}$) this implies first that $(M)_{A \sqrt{\tau}} 
  \subset (\calD_{\tau^+})_{a \sqrt{\tau^+}}$ hence also 
  \begin{equation} \label{final_proof_inclusion}
   (M)_{A \sqrt{t}} \subset (\calD_{t})_{a \sqrt{t}}
  \end{equation}
  for every $\tau \leq t \leq \tau^+$. But then for such $ \tau \leq t \leq \tau^+$, proposition $\ref{main_proposition}$ applied this time on the domain $\calD_t \times [0, t]$ yields the estimates $(\ref{final_proof_estimates})$
  for $(x,t) \in (\calD_t)_{a \sqrt{t}}$ thus on $(M)_{A \sqrt{t}}$. This, together with $(\ref{final_proof_inclusion})$, is a contradiction with the definition of $\tau$.

 \end{proof}

\appendix

\section{Appendix A - Local minimum principles} \label{appendix_local}

In this appendix we recall the ideas of the proof of theorem $\ref{ricci_pinching}$. The analogue of Hamilton-Ivey's pinching inequality for the lowest value of the Ricci tensor is due to Z.H. Zhang and is proved in 
\cite{Zhang_2015}. The local version stated in theorem $\ref{ricci_pinching}$ is obtained by transposing verbatim the argument used by B. L. Chen in \cite{Chen_2013} to establish a local version of the usual Hamilton-Ivey estimate. 

\bigskip

Minimum principles for heat-type equations are in general non local in nature. For instance, we know that non-negative Ricci curvature is preserved by any three-dimensional complete Ricci flow, yet it is easy to produce an example 
with non-negative Ricci curvature in an arbitrary ball around a point $x_0$, such that the lowest eigenvalue of the Ricci tensor nevertheless ends toward $-\infty$ at $x_0$ for arbitrarily short time when evolved
by the Ricci flow. The neckpinch example on $\setS^2 \times \setR$ depicted below suggests this

\begin{center}
 \includegraphics{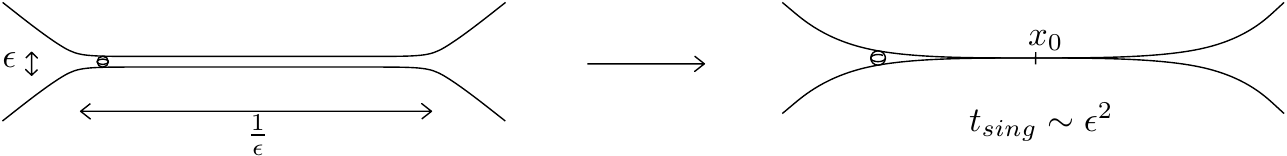}
\end{center}

If $g(t)$ is a Ricci flow on a domain $\calU \subset M^3 \times [0,T]$, we define, at each point $(x,t) \in \calU$ where $\lambda_1(x,t)$, the lowest eigenvalue of the Ricci tensor, is negative,
$u(x,t) = \frac{\scal}{- \lambda_1} - \ln(- \lambda_1 (1+t)) + 3$, while we put $u(x,t) = + \infty$ at points where $\lambda_1(x,t) \geq 0$. 

\begin{proposition} \label{proposition_pinching_equation}
 At every point where $u<0$ one has
 \begin{equation} \label{equation_pinching}
  (\del_t - \Delta) u \geq \frac{u^2}{2(1+t)}.
 \end{equation}
\end{proposition}

Now a non-linear equation of the form $(\ref{equation_pinching})$ is amenable to a localized version of the minimum principle. Consider indeed a smooth cut-off function $\Phi(x,t)$ with values in $[0,1]$, support in $\calU$
with $\Phi \equiv 1$ on a domain $\calV \subset \calU$. If one considers $v= \Psi u$, then at a minimum $(\ux,\ut)$ of $v$ on $\calU$ (assuming $v(\ux,\ut) < 0$ and $\ut>0$), one has
\[ 0 \geq (\del_t - \Delta) v \geq ( \del_t \Phi - \Delta \Phi + 2 \frac{\normes{ \nabla \Phi }^2}{\Phi} ) u + \frac{u^2}{2(1+t)}  \]
Thus if one chooses $\Phi$ such that an upper bound
\[ \del_t \Phi - \Delta \Phi + 2 \frac{\normes{ \nabla \Phi }^2}{\Phi} \leq \calC(\Phi) \]
holds on $\calU$, one gets $0 \geq \left( \calC(\Phi) + \frac{v}{2(1+t)} \right) v$, whence $v(\ux,\ut) \geq - 2 (1+t) \calC(\Phi)$. In particular, one deduces
\begin{equation} \label{minoration_2}
 u(x,t) \geq - \min( \inf u(x,0), 2 (1+t) \calC(\Phi))
\end{equation}
on $\calV$.

\bigskip

Now constructing good cut-off functions under $\frac{K}{t}$ curvature bounds is standard. Choose a smooth function $\phi: \setR \rightarrow [0,1]$ such that $\phi \equiv 1 $ on $]-\infty,0]$, $\phi \equiv 0$ on $[1,+\infty[$ and 
such that $\phi'' + 2 \frac{\phi'^2}{\phi} \leq 100$.

\begin{lemma}\label{lemme_cutoff}
  Let $g(t)$ be a Ricci flow on $M^n \times [0,T]$ and $x_0 \in M$ such that the following holds
 \begin{align*}
  & B_t(x, a+ A) \text{ is relatively compact } M, \\
  & \normes{ \riem g(t) } \leq \frac{K}{t} \text{ for } x \in B_t(x_0,\sqrt{\frac{t}{K}}).
 \end{align*}
 Then the smooth function defined by
 \[ \Phi(x,t) = \phi \left( \frac{d_t(x_0,x)+\frac{10 (n-1)}{3} \sqrt{K t} - a}{A} \right) \]
 satisfies the following properties
 \begin{align*}
  & \Phi(x,t) = 1 \text{ si } d_t(x_0,x) \leq a-\frac{10 (n-1)}{3} \sqrt{K t}, \\
  & \Phi(x,t)=0 \text{ si } d_t(x_0,x) \geq a+ A, \\
  & \del_t \Phi - \Delta \Phi + 2 \frac{\normes{ \nabla \Phi }^2}{\Phi} \leq \frac{100}{A^2}.
 \end{align*}
\end{lemma}

\begin{proof}
 Straightforward computation involving Lemma 27.18 in \cite{Kleiner_Lott}.
\end{proof}

With this choice of a cut-off function, $(\ref{minoration_2})$ directly yields theorem $\ref{ricci_pinching}$.

\begin{proof}[Proof of proposition \ref{proposition_pinching_equation}]
 We know that in an appropriate coordinate system, the curvature operator $\calM(x,t)$ satisfies the equation
 \[ \del_t \calM = ^t \Delta \calM + F(\calM) \]
 where 
 \[F (M) = Q \left( \begin{matrix} \lambda^2 + \mu \nu & & \\ & \mu^2 + \lambda \nu & \\ & & \nu^2 + \lambda \mu \end{matrix} \right) Q^{-1} \]
 if $M=Q \left( \begin{matrix} \lambda & & \\ & \mu & \\ & & \nu \end{matrix} \right) Q^{-1}$. From this it is easy to see that if we define $\calN$ from $\calM$ by
 \[ \calN(x,t) = Q \left( \begin{matrix} \frac{\lambda+\mu}{2} & & \\ & \frac{\lambda+\nu}{2} & \\ & & \frac{\mu+\nu}{2} \end{matrix} \right) Q^{-1} \]
 where $\calM=Q \left( \begin{matrix} \lambda & & \\ & \mu & \\ & & \nu \end{matrix} \right) Q^{-1}$, (i.e. $\calN$ is the matrix of the Ricci tensor in these coordinates), then it satisfies
 \[ \del_t \calN = ^t \Delta \calN + G(\calN) \]
 where 
 \[G (N) = Q \left( \begin{matrix} (\gothm+\gothn) \gothl + (\gothm-\gothn)^2 & & \\ & (\gothl+\gothn) \gothm + (\gothl-\gothn)^2 & \\ & & (\gothl+\gothm) \gothn +(\gothl-\gothm)^2 \end{matrix} \right) Q^{-1} \]
 if $N=Q \left( \begin{matrix} \gothl & & \\ & \gothm & \\ & & \gothn \end{matrix} \right) Q^{-1}.$
 
 \bigskip
 
 Thus in term of the function $u$, a straightforward computation yields
 \[ (\del_t -\Delta) u  = \frac{(\gothl-\gothm)^2 (\gothl+\gothm-\gothn)+2 (-\gothn)^3}{\gothn^2} - \frac{1}{1+t}, \]
 where $\gothl \geq \gothm \geq \gothn$ are the eigenvalues of $\calN$. In the case when  $\gothl+\gothm-\gothn \geq 0$ one gets
 \begin{equation} \label{minoration_1}
  (\del_t - \Delta) u \geq - \gothn - \frac{1}{1+t}
 \end{equation}
 whereas if $\gothl+\gothm-\gothn <0$, then $\gothm<\gothl<0$ hence $(\gothl-\gothm)^2 \leq  \gothm^2$. Since $\gothl+\gothm-\gothn \geq \gothm$,
 \[ (\del_t-\Delta) u \geq \frac{\gothm^3 + 2 (-\gothn)^3}{(-\gothn)^2} - \frac{1}{1+t} \]
 whence $(\ref{minoration_1})$ in this case too. Writing $\ln (-\gothn)(1+t) = - u + \frac{\gothl+\gothm+\gothn}{-\gothn} + 3$, and using $\frac{\gothl+\gothm+\gothn}{-\gothn} \geq 3 \gothn$, one gets from $(\ref{minoration_1})$
 \[ (\del_t-\Delta) u \geq \frac{1}{1+t} ( (-\gothn)(1+t) - 1 ) \geq \frac{e^{-u}-1}{1+t}. \]
 whence the proposition.
 
\end{proof}

\section{Appendix B - Classical results for flows of bounded curvature}

The following proposition is theorem 3.1 in \cite{Chen_2009}:

\begin{proposition} \label{misc_0}
 There exists $\Lambda(n,C,K)$ with the following property. If $(U,g(t))$ is a Ricci flow on $[0,T]$ such that 
 \begin{align*}
  & B_t(x_0,1) \text{ is relatively compact in $U$}, \\
  & \normes{ \riem g(x,t) } \leq \frac{K}{t},
 \end{align*}
 for $x \in B_t(x_0,1)$ and $0 < t \leq T$, and such that at initial time
 \[ \normes{ \riem g(x,0) } \leq C \]
 on $B_0(x_0,1)$ then
 \[ \normes{ \riem g(x,t) } \leq \frac{\Lambda}{(1-d_t(x_0,x))^2}, \] 
 for $x \in B_t(x_0, 1)$, $0 \leq t  \leq T$.
\end{proposition}

This can be refined into the following ``non-doubling time'' property (this is essentially 2.4 in \cite{Xu_2015})

\begin{corollary} \label{misc_1}
 There exists $\epsilon_{\ref{misc_1}}(n,K)$ with the following property. If $(U,g(t))$ is a Ricci flow such that 
 \begin{align*}
  & B_t(x_0,1) \text{ is relatively compact in $U$}, \\
  & \normes{ \riem g(x,t) } \leq \frac{K}{t},
 \end{align*}
 for $x \in B_t(x_0,1)$ and $0 < t \leq \epsilon_{\ref{misc_1}}^2$, and such that at initial time
 \[ \begin{array}{cc} \normes{ \riem g(x,0) } \leq 1, & \inj{g(0)} x \geq 1, \end{array} \]
 on $B_0(x_0,1)$ then
 \[ \begin{array}{cc} \normes{ \riem g(x,t) } \leq 4, & \inj{g(t)} x \geq \frac{1}{2}, \end{array} \] 
 on $B_t(x_0, \epsilon_{\ref{misc_1}})$ for $0 \leq t  \leq \epsilon_{\ref{misc_1}}^2$.
\end{corollary}

For complete flows of bounded curvature, $\frac{1}{t}$ curvature bound required in the above statements can be obtained by Perelman's pseudo-locality theorem (see \cite{Perelman_2002}, 10 \cite{Kleiner_Lott}, 30 for details, 
and \cite{Chau_2011} for the extension to complete non-compact data),

\begin{corollary} \label{misc_2}
 There exists $\epsilon_{\ref{misc_2}}(n)$ with the following property. Let $(M,g(t))$ be a complete Ricci flow on $[0,T[$ such that
 \[ \sup_{M \times [0,T[} \normes{ \riem g(x,t) } < + \infty \]
 and $x_0 \in M$ be such that at initial time
 \[ \begin{array}{cc} \normes{ \riem g(x,0) } \leq 1, & \inj{g(0)} x \geq 1, \end{array} \]
 on $B_t(x_0,1)$ then
 \[ \begin{array}{cc} \normes{ \riem g(x,t) } \leq 4, & \inj{g(t)} x \geq \frac{1}{2}, \end{array} \] 
 on $B_t(x_0, \epsilon_{\ref{misc_2}})$ for $0 \leq t  \leq \epsilon_{\ref{misc_2}}^2$. 
\end{corollary}

\bibliographystyle{plain}

\bibliography{Biblio.bib}

\end{document}